%% file: main.tex
\documentclass[a4paper, reqno, 11pt]{amsart}
\usepackage[utf8]{inputenc}

\usepackage{amsthm}
\usepackage{amssymb}
\usepackage{exscale}
\usepackage[mathscr]{euscript}
\usepackage{url}
\usepackage[all,ps,tips,tpic]{xy}
\usepackage{graphicx}
\usepackage{color}

\usepackage{booktabs}
\usepackage{tabularx}

\newenvironment{altenumerate}
   {\begin{list}
      {\textup{(\theenumi)} }
      {\usecounter{enumi}
       \setlength{\labelwidth}{0pt}
       \setlength{\labelsep}{2pt}
       \setlength{\leftmargin}{0pt}
       \setlength{\itemsep}{\the\smallskipamount}
       \renewcommand{\theenumi}{\roman{enumi}}
      }}
   {\end{list}}

\newenvironment{altenumeratealpha}
   {\begin{list}
      {\textup{(\theenumi)} }
      {\usecounter{enumi}
       \setlength{\labelwidth}{0pt}
       \setlength{\labelsep}{2pt}
       \setlength{\leftmargin}{0pt}
       \setlength{\itemsep}{\the\smallskipamount}
       \renewcommand{\theenumi}{\alph{enumi}}
      }}
   {\end{list}}

\newenvironment{altitemize}
   {\begin{list}
      {$\bullet$ }
      {\setlength{\labelwidth}{0pt}
       \setlength{\labelsep}{2pt}
       \setlength{\leftmargin}{0pt}
       \setlength{\itemsep}{\the\smallskipamount}
      }}
   {\end{list}}

\newtheorem{lem}{Lemma}[section]
\newtheorem{definition}[lem]{Definition}

\theoremstyle{remark}

\newtheorem{example}[lem]{Example}

\theoremstyle{plain}
\newtheorem{theorem}[lem]{Theorem}
\newtheorem{lemma}[lem]{Lemma}
\newtheorem{proposition}[lem]{Proposition}
\newtheorem{corollary}[lem]{Corollary}

\theoremstyle{remark}
\newtheorem{remark}[lem]{Remark}

\newcommand{\Z}{\mathbb{Z}}
\newcommand{\Q}{\mathbb{Q}}

\newcommand{\Zp}{\Z_p}
\newcommand{\Ztwo}{\Z_2}
\newcommand{\Qtwo}{\Q_2}

\newdir{ >}{{}*!/-5pt/@{>}}
\SelectTips{cm}{10}

\newenvironment{psmallmatrix}{\left(\begin{smallmatrix}}{\end{smallmatrix}\right)}

\begin{document}

\title[Elementary local densities via lifting]{Elementary local representation densities at all primes via lifting recursions}

\author{Samuel Griffiths}
\address{Auckland University of Technology}
\email{xtb2822@autuni.ac.nz}
\date{}

\subjclass[2020]{11E12, 11E20}
\keywords{local representation density, Siegel series, hyperbolic plane, dyadic local density, quadratic forms, lifting recursion, Hensel lifting}

\begin{abstract}
Let $p$ be a prime and let $L$ be a quadratic $\Zp$-lattice with quadratic form $Q$.
For $t\ne 0$ the local representation density $\alpha_p(t;L)$ is the stable normalised growth of the congruence counts of solutions to
$Q(v)\equiv t\pmod{p^m}$.
We compute these counts and densities explicitly for the hyperbolic plane $H_0$ over $\Zp$, uniformly in the prime~$p$, and at $p=2$ for the basic dyadic blocks of Definition~\ref{def:dyadic-base-blocks} (rank-$1$ Type~I blocks and the even binary planes $2^aH_\varepsilon$), together with the anisotropic ternary lattice $L_3=\langle 2\rangle^{\oplus 3}$.

At the dyadic prime the usual Jacobian/Hensel lifting mechanism breaks down in the bilinear-lattice convention $Q(v)=\langle v,v\rangle$.
The main new input is an explicit \emph{half-lift} involution for diagonal sums of squares, which yields a stable lifting recursion with factor $2^{d-1}$ under the primitivity hypothesis $4\nmid a$.
As applications we obtain closed forms for the three-squares congruence counts (hence $\alpha_2(t;L_3)$) and a prime-uniform formula for the densities of the scaled hyperbolic planes $p^eH_0$ in the standard normalisation $q=\langle\cdot,\cdot\rangle/2$.
\end{abstract}

\maketitle


\input{sections/01_introduction}
\input{sections/02_preliminaries_dyadic_jordan}
\input{sections/03_hyperbolic_plane_counts}
\input{sections/04_anisotropic_even_plane}
\input{sections/05_half_lift_and_L3}
\input{sections/06_explicit_density_formulas}

\appendix
\input{sections/A_typeI_rank1_blocks}

\bibliographystyle{abbrv}
\bibliography{references}

\end{document}

%% file: sections/01_introduction.tex
\section{Introduction}
\label{sec:intro}

Local representation densities of quadratic lattices are the local factors in many global counting problems.
They appear in the Siegel series and Siegel's mass formula, and they control Fourier coefficients of theta series and Eisenstein series.
In concrete arithmetic applications one is often forced to evaluate these local factors for specific lattices, and the dyadic prime is typically the main obstacle.

At odd primes, once the relevant solution locus is smooth modulo $p$, Hensel lifting implies that reduction maps between solution sets have constant fibre size $p^{n-1}$, and the density stabilises quickly.
At the dyadic prime this mechanism can break down even for very simple lattices.
In the bilinear-lattice convention $Q(v)=\langle v,v\rangle$, the linearisation of $Q-t$ at a solution $\bar v$ is the functional $h\mapsto 2\langle \bar v,h\rangle$, which vanishes modulo~$2$.
Thus the standard Jacobian criterion yields no lifting information.

Every nondegenerate quadratic lattice over $\Ztwo$ admits a Jordan decomposition into scaled unimodular constituents (Proposition~\ref{prop:jordan}).
Accordingly, explicit formulas for the basic dyadic blocks are a natural starting point for dyadic computations.
In this paper we determine the finite-level representation counts and local densities for the basic dyadic blocks of Definition~\ref{def:dyadic-base-blocks} (rank-$1$ Type~I blocks and the even planes $2^aH_\varepsilon$), together with the anisotropic ternary lattice $L_3$.

The main new input is a very simple involution argument, the \emph{half-lift}.
For diagonal sums of squares $Q_d=x_1^2+\cdots+x_d^2$ and a primitive target $a$ with $4\nmid a$, translation by $2^{n-1}u$ in a primitive direction produces a free involution on the solution set modulo $2^n$ whose orbits control lifts to level $2^{n+1}$.
In each orbit exactly one residue class lifts, and whenever a class lifts, all lifts in its fibre are solutions.
This yields the stable recursion $N_{d,n+1}(a)=2^{d-1}N_{d,n}(a)$ for $n\ge 3$ (Theorem~\ref{thm:half-lift-principle}).
For $d=3$ it leads to closed forms for the three-squares congruence counts and hence an explicit formula for $\alpha_2(t;L_3)$.

\smallskip
We fix the conventions for the finite-level representation counts $r_m(t;L)$ and the local density $\alpha_p(t;L)$ in \S\ref{sec:prelim}.
At $p=2$ and for even lattices there is a second common normalisation $q:=Q/2$; we work with $Q$ unless explicitly indicated, and Lemma~\ref{lem:q-dictionary} gives the precise conversion.

General formulas for the dyadic Siegel series are known, expressed for instance in terms of Gross--Keating invariants or group-scheme methods (see for example
\cite{grossKeating1993,choYamauchi2020,ganYu2000,ikedaKatsurada2018,kitaokaBook}).
The densities computed here agree with these formulas (cf.\ Yang's explicit expressions~\cite{yang1998}), but this serves only as an \emph{external consistency check} and is not used as an input to any argument.

At odd primes, the hyperbolic-plane computations of \S\ref{sec:hyperbolic} are classical; we include a self-contained counting argument to keep the exposition uniform and to make the prime-uniform statement of Theorem~\ref{thm:prime-uniform-h0} transparent.
The genuinely new feature of this paper is dyadic: the half-lift involution of Theorem~\ref{thm:half-lift-principle} and the resulting elementary derivation of the dyadic block densities without appealing to the full Siegel-series machinery.

All computations in the body of the paper are organised around the same elementary theme:
describe the fibres of the reduction maps $L/2^{m+1}L\to L/2^mL$ on the relevant solution sets.
Concretely,
\begin{altenumerate}
\item for the split even plane $H_0$ this reduces to valuation-stratified counting for the bilinear congruence $2xy\equiv t$ (see Lemmas~\ref{lem:2xy-odd} and~\ref{lem:2xy-2});
\item for the anisotropic plane $H_1$ it becomes a lifting recursion for a norm congruence in the unramified quadratic extension of~$\Qtwo$ (see Lemma~\ref{lem:H1});
\item for the anisotropic ternary block $L_3$ one needs an additional involution (the half-lift) because the naive dyadic linearisation carries no information (see Theorem~\ref{thm:half-lift-principle}).
\end{altenumerate}

This paper focuses on the basic dyadic blocks (Definition~\ref{def:dyadic-base-blocks}) and the anisotropic ternary lattice $L_3$.
Once the basic block counts are known, representation counts for orthogonal sums reduce to finite convolutions (Proposition~\ref{prop:orthogonal-convolution}).
We do not attempt closed forms for general orthogonal sums.

\smallskip
\noindent
For quick reference, Table~\ref{tab:block-densities} in \S\ref{sec:densities} lists the final density values.
Readers primarily interested in the dyadic novelty may start with \S\ref{sec:half-lift}.
\subsection{Main results}

Let $p$ be a prime.
Section~\ref{sec:hyperbolic} treats the hyperbolic plane $H_0$ over $\Zp$ for every $p$ and gives explicit formulas for the congruence counts $M_{p,m}(s)$ in~\eqref{eq:2xy}, together with their generating series and stable-range behaviour.
Section~\ref{sec:densities} records the resulting density formulas for $\alpha_p(t;H_0)$ and the basic dyadic blocks.
In particular, after passing to the standard normalisation $q=\langle\cdot,\cdot\rangle/2$ at $p=2$ (Lemma~\ref{lem:q-dictionary}), the density of $H_0$ and its scalings admits a single prime-uniform expression (Theorem~\ref{thm:prime-uniform-h0}).
The remaining sections specialise to $p=2$ and evaluate the basic dyadic blocks.

We therefore recall the dyadic Jordan decomposition (Proposition~\ref{prop:jordan}) and the basic dyadic blocks (Definition~\ref{def:dyadic-base-blocks}) that will be treated explicitly.
In addition to the rank-$2$ even blocks, the rank-$3$ lattice
\begin{equation}\label{eq:L3}
  L_3 := \langle 2\rangle^{\oplus 3}\qquad\text{with }Q(x)=2(x_1^2+x_2^2+x_3^2)
\end{equation}
plays a distinguished role: it is the simplest anisotropic even ternary lattice over $\Ztwo$, and its representation counts amount to the classical three-squares congruence problem.
Indeed, if $x_1^2+x_2^2+x_3^2=0$ in $\Ztwo$, then reducing modulo $8$ forces $x_1,x_2,x_3$ to be even, and iterating gives $x_1=x_2=x_3=0$; hence $L_3$ is anisotropic over $\Qtwo$.

Our first main theorem is a lifting principle for diagonal sums of squares in arbitrary dimension, under the mild primitivity hypothesis $4\nmid a$.

\begin{theorem}[Half-lift principle for sums of squares]\label{thm:half-lift-principle}
Let $d\ge 1$ and set
\[
  Q_d(x_1,\dots,x_d):=x_1^2+\cdots+x_d^2.
\]
For $n\ge 1$ and $a\in\Z$ let
\[
  N_{d,n}(a):=\#\{x\in (\Z/2^n\Z)^d: Q_d(x)\equiv a\pmod{2^n}\}.
\]
Assume $n\ge 3$ and $4\nmid a$.
Then
\[
  N_{d,n+1}(a)=2^{d-1}\,N_{d,n}(a).
\]
Equivalently, among the residue classes modulo $2^n$ satisfying $Q_d\equiv a$, exactly half admit lifts to level $2^{n+1}$; moreover, whenever a class lifts, \emph{all} $2^d$ lifts in its fibre are solutions (cf.\ Lemma~\ref{lem:fibre-invariance}).
\end{theorem}

\begin{remark}[Scope of the half-lift involution]\label{rem:half-lift-scope}
The involution argument proving Theorem~\ref{thm:half-lift-principle} uses the specific diagonal form
$Q_d=x_1^2+\cdots+x_d^2$ and the primitivity hypothesis $4\nmid a$.
The primitivity is essential in general: for example, in dimension $d=4$ one checks directly that
\(N_{4,4}(8)=1536\) while \(N_{4,3}(8)=128\), so the ratio is $12\ne 2^{d-1}=8$.
On a conceptual level, two features of the proof break down once these hypotheses are relaxed:
if $4\mid a$ then solutions modulo $2^n$ may have all coordinates even, so the first step (which extracts an odd coordinate to choose an involution direction) can fail; and for $n=2$ the quadratic error term in the expansion \eqref{eq:expansion-d} no longer vanishes modulo $2^{n+1}$, so the translation $x\mapsto x+2^{n-1}u$ need not toggle $Q_d(x)$ by a clean additive $2^n$.
We do not claim a dimension-free half-lift recursion for arbitrary diagonal unit coefficients (or for general dyadic lattices) without further reduction.
In the dyadic setting, the half-lift mechanism is used here only to analyse the even ternary lattice $L_3$ in \eqref{eq:L3};
rank-$1$ Type~I blocks are handled separately via explicit square-root counts (Appendix~\ref{app:typeI}).
\end{remark}

For $d=3$, Theorem~\ref{thm:half-lift-principle} combines with a simple $4$-adic descent and a base computation modulo $8$ to yield closed forms for the three-squares counts once one reaches the stable range $n\ge 2k+3$ for $a=4^k a_0$ with $4\nmid a_0$ (and hence an explicit formula for $\alpha_2(t;L_3)$).
More generally, in the stable range the half-lift recursion has growth factor $2^{d-1}$, which is exactly the ``smooth'' Hensel factor $p^{d-1}$ specialised to $p=2$ (Remark~\ref{rem:half-lift-hensel}).
We also record explicit formulas for the even rank-$2$ Jordan blocks (hyperbolic and anisotropic).
As a consequence we obtain explicit densities $\alpha_2(t;B)$ for each basic dyadic block $B$ of Definition~\ref{def:dyadic-base-blocks}.

\begin{theorem}[Prime-uniform density for the hyperbolic plane]
\label{thm:prime-uniform-h0}
Let $p$ be a prime and let $e\ge 0$.
Let $H_0$ denote the rank-$2$ hyperbolic plane over $\Zp$ with bilinear form
\(\langle (x,y),(x',y')\rangle = xy'+x'y\).
On the scaled lattice $p^eH_0$ consider the quadratic form
\[
  q_e(v):=\tfrac12\langle v,v\rangle = p^e xy\in\Zp.
\]
For $m\ge 1$ and $t\in\Zp$ set
\[
  r^{(q)}_{m}(t;p^eH_0)
  :=\#\{(x,y)\in(\Z/p^m\Z)^2 : p^e xy\equiv t\!\!\pmod{p^m}\},
\]
and define the associated density (when the limit exists) by
\[
  \alpha_p^{(q)}(t;p^eH_0)
  :=\lim_{m\to\infty}p^{-m}\,r^{(q)}_{m}(t;p^eH_0).
\]
Let $t\in\Zp\setminus\{0\}$ and write $v:=v_p(t)$.
Then
\[
  \alpha_p^{(q)}(t;p^eH_0)=
  \begin{cases}
    0, & v<e,\\[2pt]
    (v-e+1)(p-1)p^{e-1}, & v\ge e.
  \end{cases}
\]
\smallskip
\noindent In particular, under the $q=\langle\cdot,\cdot\rangle/2$ normalisation, both the vanishing threshold and the density formula for the scaled hyperbolic plane are uniform for all primes $p$, including $p=2$.
\end{theorem}

Theorem~\ref{thm:prime-uniform-h0} is a formal consequence of the $Q$-normalised density formulas for the hyperbolic plane (Corollaries~\ref{cor:density-hyp-scaled-odd} and~\ref{cor:density-hyp-scaled}) together with the conversion Lemma~\ref{lem:q-dictionary}; we give the short proof in \S\ref{sec:densities}.

\subsection*{Organisation}
Section~\ref{sec:prelim} fixes notation and recalls the dyadic Jordan decomposition and the basic dyadic blocks.
Section~\ref{sec:hyperbolic} treats the hyperbolic plane $H_0$ uniformly for all primes $p$ by explicit congruence counting, including closed forms for its stable-range counts and generating series.
Section~\ref{sec:anisotropic} treats the anisotropic even plane; here one can recast the counting problem as a norm-congruence in an unramified quadratic extension, where the unit locus admits a clean lifting recursion.
Section~\ref{sec:half-lift} proves the half-lift principle and derives closed forms for the anisotropic ternary lattice $L_3$.
Section~\ref{sec:densities} summarises the resulting density formulas.
Appendix~\ref{app:typeI} records the rank-$1$ (Type~I) dyadic block counts.

\subsection*{Computational checks}

The arguments in this paper are self-contained and elementary, and they do not rely on computation.
As an external consistency check, the arXiv source bundle includes ancillary files containing finite-level verifiers for the congruence counts and a Lean~4 formalisation of the half-lift involution (see \texttt{anc/verify.py} and the directory \texttt{anc/lean/}; build instructions are in \texttt{anc/README\_CODE.txt}).


%% file: sections/02_preliminaries_dyadic_jordan.tex
\section{Preliminaries and dyadic Jordan blocks}
\label{sec:prelim}

We fix general conventions for finite-level representation counts and local representation densities, and then recall the dyadic Jordan blocks at $p=2$.
We also record the relation between the two standard dyadic normalisations $Q$ and $q=Q/2$ on even lattices.

\subsection{Representation counts and densities}
\label{subsec:counts-densities}

For a prime $p$ we write $v_p(\cdot)$ for the $p$-adic valuation on $\Zp$, normalised by $v_p(p)=1$ and $v_p(0)=\infty$.
When working modulo $p^m$, for a residue class $s\in\Z/p^m\Z$ we also write
\(v_p(s)\in\{0,1,\dots,m\}\) for the largest exponent $v$ such that $s\equiv 0\pmod{p^v}$ (with the convention $v_p(0)=m$).

For $t\in\Zp\setminus\{0\}$ we will often write $t=p^{v_p(t)}u$ with $u\in\Zp^\times$.

A \emph{quadratic $\Zp$-lattice} is a free $\Zp$-module $L$ of finite rank equipped with a symmetric bilinear form $\langle\cdot,\cdot\rangle\colon L\times L\to\Zp$.
We write
\[Q(v):=\langle v,v\rangle\in\Zp\]
for the associated quadratic form.

Many authors instead start from an integral quadratic form $q$ with associated bilinear form
$B_q(x,y)=q(x+y)-q(x)-q(y)$. If $p$ is odd and $\langle\cdot,\cdot\rangle=B_q$, then $Q(v)=\langle v,v\rangle=2q(v)$, so the choice between $Q$ and $q$ is immaterial up to scaling the target by the unit~$2$.
At $p=2$ this distinction becomes substantive on even lattices, and we record an exact dictionary in \S\ref{subsec:q-convention}.

For $m\ge 1$ and $t\in\Zp$ define the finite-level representation count
\begin{equation}\label{eq:rep-count}
  r_m(t;L) := \#\bigl\{v\in L/p^mL : Q(v)\equiv t\pmod{p^m}\bigr\}.
\end{equation}
When $t$ is nonzero, the \emph{local representation density} is recovered as
\begin{equation}\label{eq:density-def}
  \alpha_p(t;L)
  := \lim_{m\to\infty} p^{-m(n-1)}\,r_m(t;L),\qquad n:=\mathrm{rank}_{\Zp}(L),
\end{equation}
whenever the limit exists.
Throughout the paper we restrict to \emph{nonzero targets} $t\ne 0$.
The case $t=0$ (the singular value in the Siegel-series terminology) behaves differently; see Remark~\ref{rem:singular-targets}.

\subsection{Dyadic Jordan blocks}

We now specialise to $p=2$.
A \emph{quadratic $\Ztwo$-lattice} is a quadratic $\Zp$-lattice for $p=2$, i.e.\ a free $\Ztwo$-module $L$ of finite rank equipped with a symmetric bilinear form $\langle\cdot,\cdot\rangle\colon L\times L\to\Ztwo$.
As above we work with the bilinear-lattice convention $Q(v)=\langle v,v\rangle$.

\begin{definition}[Basic dyadic blocks]\label{def:dyadic-base-blocks}
For the explicit computations in this paper we single out the following quadratic $\Ztwo$-lattices; we refer to them as the \emph{basic dyadic blocks}:
\begin{altenumerate}
\item \textbf{Type~I (rank $1$).} The lattice $\langle 2^a u\rangle$ with Gram matrix $[2^a u]$, where $a\ge 0$ and $u\in\Ztwo^\times$.
\item \textbf{Type~II (rank $2$, even).} The lattice $2^a H_\varepsilon$ with Gram matrix $2^a H_\varepsilon$, where $a\ge 0$ and
\[
  H_0:=\begin{pmatrix}0&1\\1&0\end{pmatrix}\qquad\text{and}\qquad
  H_1:=\begin{pmatrix}2&1\\1&2\end{pmatrix}.
\]
\end{altenumerate}
Only the square class of the unit $u$ (equivalently $u\bmod 8$) will matter for solvability in the rank-one case.
\end{definition}

\begin{remark}\label{rem:evenness-dyadic}
We use ``even'' in the quadratic sense: $Q(L)\subset 2\Ztwo$.
Equivalently (after choosing a $\Ztwo$-basis), the Gram matrix has all diagonal entries in $2\Ztwo$.
In rank~$2$ and determinant a unit, there are exactly two isomorphism classes of such even unimodular lattices over~$\Ztwo$, represented by $H_0$ (hyperbolic/split) and $H_1$ (anisotropic).
Scaling by $2^a$ produces all even binary Jordan blocks.
\end{remark}

\begin{proposition}[Jordan decomposition]\label{prop:jordan}
Let $L$ be a nondegenerate quadratic lattice over $\Ztwo$.
Then there exists an orthogonal decomposition
\[
  L \;\cong\; \perp_{a\ge 0} 2^a L_a,
\]
with each $L_a$ a unimodular quadratic $\Ztwo$-lattice (possibly $0$).
Moreover, the multiset of integers $a$ together with the multiplicities $\mathrm{rank}(L_a)$ (equivalently, the Jordan \emph{scales} and \emph{ranks}) is uniquely determined by $L$.
\end{proposition}

\begin{remark}\label{rem:jordan-background}
We treat Proposition~\ref{prop:jordan} as background structure; detailed proofs can be found in standard references on integral quadratic forms, e.g.\ \cite[\S93]{omearaQuad} or \cite[Ch.~1]{kitaokaBook}.
For $\Ztwo$-lattices, a Jordan splitting as in Proposition~\ref{prop:jordan} is generally \emph{not} unique on the nose at the level of the unimodular constituents $L_a$; what is canonical is the associated Jordan symbol/type (and hence the scales, ranks, and the other local invariants). In this paper we use only the existence of some splitting and, in low ranks, the standard classification into the basic blocks below.
For the explicit computations in this paper we focus on the basic dyadic blocks of Definition~\ref{def:dyadic-base-blocks}.
In particular, any binary quadratic lattice over $\Ztwo$ decomposes into an orthogonal sum of scaled rank-$1$ blocks $\langle 2^a u\rangle$ and scaled even planes $2^aH_\varepsilon$; cf.\ the same references.
Once the basic block counts are known, representation counts for orthogonal sums reduce to finite convolutions (Proposition~\ref{prop:orthogonal-convolution}).
\end{remark}

\subsection{The $q=Q/2$ convention on even dyadic lattices}
\label{subsec:q-convention}

When $p=2$ and $L$ is even, a common alternative normalisation is to take $q:=Q/2$ as the primary quadratic form.
For convenience we record a dictionary between the two conventions.

\begin{lemma}[Dictionary with the $q=\langle\cdot,\cdot\rangle/2$ convention]\label{lem:q-dictionary}
Assume $L$ is \emph{even}, i.e.\ $Q(L)\subset 2\Ztwo$, and set $q(v):=Q(v)/2\in\Ztwo$.
For $m\ge 1$ define
\[
  r_m^{(q)}(t;L):=\#\{v\in L/2^mL : q(v)\equiv t\!\!\pmod{2^m}\}.
\]
Then for $t\in\Ztwo$ and $n:=\mathrm{rank}_{\Ztwo}(L)$:
\begin{altenumerate}
\item If $t$ is odd, then $r_m(t;L)=0$ for all $m\ge 1$.
\item If $t=2t'$ is even, then for every $m\ge 2$ one has
\[
  r_m(t;L)=2^{n}\,r_{m-1}^{(q)}(t';L).
\]
In particular, $r_1(2t';L)=2^n$.
\end{altenumerate}
Consequently, whenever the limits exist,
\[\alpha_2(2t';L)=2\,\alpha_2^{(q)}(t';L),\qquad
  \alpha_2^{(q)}(t';L):=\lim_{m\to\infty}2^{-m(n-1)}\,r_m^{(q)}(t';L).\]
\end{lemma}

\begin{proof}
If $t$ is odd, then $Q(v)$ is even for all $v\in L$ and there are no solutions.

Now suppose $t=2t'$.
For $m\ge 2$ one has
\[
  Q(v)\equiv 2t'\pmod{2^m}
  \quad\Longleftrightarrow\quad
  q(v)\equiv t'\pmod{2^{m-1}}.
\]
The right-hand condition depends only on the class of $v$ modulo $2^{m-1}L$.
Each class in $L/2^{m-1}L$ has exactly $2^n$ lifts to $L/2^mL$, and all lifts have the same $q$-value modulo $2^{m-1}$.
Thus $r_m(2t';L)=2^n r_{m-1}^{(q)}(t';L)$ for $m\ge 2$.

For $m=1$ the congruence is modulo $2$ and is automatic, hence $r_1(2t';L)=\#(L/2L)=2^n$.
The density identity follows by dividing by $2^{m(n-1)}$ and taking $m\to\infty$.
\end{proof}

\begin{remark}[On the $Q$--vs.~$q$ normalisation at $p=2$]\label{rem:Q-vs-q-normalisation}
Unless explicitly indicated (as in Theorem~\ref{thm:prime-uniform-h0}), we work with the bilinear-lattice convention
$Q(v)=\langle v,v\rangle$ and the densities $\alpha_p(t;L)$ of \eqref{eq:density-def}.
When $p=2$ and $L$ is even, many sources instead take the quadratic form $q:=Q/2$ as primary.
Lemma~\ref{lem:q-dictionary} gives an exact conversion between the two normalisations, and in particular
\[\alpha_2(2t;L)=2\,\alpha_2^{(q)}(t;L)\qquad (L\ \text{even}).\]
The summary Table~\ref{tab:block-densities} records $Q$-normalised densities; Theorem~\ref{thm:prime-uniform-h0} is stated in the $q$-normalisation to make its formula uniform in~$p$.
\end{remark}


%% file: sections/03_hyperbolic_plane_counts.tex
\section{Congruence counts for the hyperbolic plane}
\label{sec:hyperbolic}

In this section we compute the level-$p^m$ congruence counts for the hyperbolic plane $H_0$.
We treat odd primes and $p=2$ separately.
We then record the scaling reduction needed for $p^eH_0$ and package the answer into a generating series.
The stable-range formulas will be converted into density formulas in \S\ref{sec:densities}.

The computation is by a valuation stratification: the number of solutions of \eqref{eq:2xy} depends only on $v_p(s)$, and each valuation stratum contributes the same number of pairs.

The hyperbolic plane $H_0$ has Gram matrix $\begin{psmallmatrix}0&1\\1&0\end{psmallmatrix}$ and quadratic form $Q(x,y)=2xy$.
For a prime $p$ and an exponent $m\ge 1$ we therefore consider the basic congruence
\begin{equation}\label{eq:2xy}
  2xy\equiv s\pmod{p^m}.
\end{equation}

For $m\ge 1$ and $s\in\Z/p^m\Z$ define
\[
  M_{p,m}(s):=\#\{(x,y)\in (\Z/p^m\Z)^2: 2xy\equiv s\!\!\pmod{p^m}\}.
\]
When $p=2$ we write $M_m(s)$ for $M_{2,m}(s)$.

\subsection*{Odd primes}

\begin{lemma}[Counting $2xy\equiv s\pmod{p^m}$ at odd primes]\label{lem:2xy-odd}
Let $p$ be an odd prime, $m\ge 1$, and $s\in\Z/p^m\Z$.
Write $v:=v_p(s)$ with the convention $v=m$ if $s\equiv 0\pmod{p^m}$.
Then:
\begin{altenumerate}
\item If $0\le v<m$, then $M_{p,m}(s)=(v+1)(p-1)p^{m-1}$.
\item If $v=m$ (i.e.\ $s\equiv 0\pmod{p^m}$), then $M_{p,m}(0)=(m(p-1)+p)p^{m-1}$.
\end{altenumerate}
\end{lemma}

\begin{proof}
Since $p$ is odd, $2\in(\Z/p^m\Z)^\times$ and 
\eqref{eq:2xy} is equivalent to
\[
  xy\equiv t\pmod{p^m},\qquad t:=2^{-1}s\in\Z/p^m\Z.
\]
The valuation $v_p(t)$ equals $v_p(s)=v$ because $2$ is a unit.

\smallskip
\noindent\emph{Case $0\le v<m$.}
Stratify by $i:=v_p(x)$.
For each $0\le i\le v$, the number of residue classes $x\bmod p^m$ with $v_p(x)=i$ is $(p-1)p^{m-i-1}$.
Write $x=p^i x'$ with $x'\in(\Z/p^{m-i}\Z)^\times$.
Then $xy\equiv t\pmod{p^m}$ becomes
\[
  x'\,y\equiv p^{v-i}t'\pmod{p^{m-i}},\qquad t':=t/p^v\in(\Z/p^{m-v}\Z)^\times.
\]
As $x'$ is a unit modulo $p^{m-i}$, there is a unique solution $y_0\bmod p^{m-i}$.
It has exactly $p^i$ lifts modulo $p^m$, so for each fixed $i$ the number of pairs with $v_p(x)=i$ is
\[(p-1)p^{m-i-1}\cdot p^i=(p-1)p^{m-1}.
\]
Summing over $i=0,1,\dots,v$ yields $M_{p,m}(s)=(v+1)(p-1)p^{m-1}$.

\smallskip
\noindent\emph{Case $v=m$ (i.e.\ $s\equiv 0\bmod p^m$).}
We count solutions to $xy\equiv 0\pmod{p^m}$.
For each $0\le i\le m-1$ there are $(p-1)p^{m-i-1}$ choices of $x$ with $v_p(x)=i$, and then $y$ must be divisible by $p^{m-i}$, giving $p^i$ choices for $y\bmod p^m$.
Thus each $i\in\{0,\dots,m-1\}$ contributes $(p-1)p^{m-1}$ pairs, for a total of $m(p-1)p^{m-1}$.
Finally, if $x\equiv 0\pmod{p^m}$ there are $p^m$ choices for $y$.
Therefore
\[
  M_{p,m}(0)=m(p-1)p^{m-1}+p^m=(m(p-1)+p)p^{m-1},
\]
as claimed.
\end{proof}

\subsection*{The dyadic prime}

\begin{lemma}[Counting $2xy\equiv s\pmod{2^m}$ at $p=2$]\label{lem:2xy-2}
Let $m\ge 1$ and $s\in\Z/2^m\Z$.
Write $v:=v_2(s)$ with the convention $v=m$ if $s\equiv 0\pmod{2^m}$.
Then:
\begin{altenumerate}
\item If $v=0$ (i.e.\ $s$ is odd), then $M_m(s)=0$.
\item If $1\le v<m$, then $M_m(s)=v\cdot 2^m$.
\item If $v=m$ (i.e.\ $s\equiv 0\pmod{2^m}$), then $M_m(0)=(m+1)2^m$.
\end{altenumerate}
\end{lemma}

\begin{proof}
\noindent\emph{Strategy.} If $s$ is even, we reduce $2xy\equiv s\pmod{2^m}$ to $xy\equiv t\pmod{2^{m-1}}$, stratify by $v_2(x)$, and then use that each solution modulo $2^{m-1}$ lifts to exactly $4$ solutions modulo $2^m$.

If $s$ is odd there are no solutions since $2xy$ is always even.

Now assume $s$ is even and write $s=2t$.
Then $2xy\equiv 2t\pmod{2^m}$ is equivalent to $xy\equiv t\pmod{2^{m-1}}$.
Each solution modulo $2^{m-1}$ has exactly $4$ lifts modulo $2^m$, obtained by independently adding $0$ or $2^{m-1}$ to $x$ and to $y$.
Thus $M_m(s)=4P_{m-1}(t)$, where for $k\ge 1$ we set
\[
  P_k(t):=\#\{(x,y)\in(\Z/2^k\Z)^2: xy\equiv t\!\!\pmod{2^k}\}.
\]
We compute $P_k(t)$ by a valuation stratification.
Write $u:=v_2(t)$ with the convention $u=k$ if $t\equiv 0\pmod{2^k}$.

\smallskip
\noindent\emph{Case 1: $t\not\equiv 0\pmod{2^k}$ (so $0\le u<k$).}
For $0\le i\le u$, the number of residues $x\bmod 2^k$ with $v_2(x)=i$ is $2^{k-i-1}$.
Writing $x=2^i x'$ with $x'$ odd modulo $2^{k-i}$, the congruence $xy\equiv t\pmod{2^k}$ becomes
\[
  x'\,y\equiv 2^{u-i}t'\pmod{2^{k-i}},\qquad t':=t/2^u\in(\Z/2^{k-u}\Z)^\times.
\]
Since $x'$ is a unit modulo $2^{k-i}$, there is a unique solution $y_0\bmod 2^{k-i}$, and it has exactly $2^i$ lifts modulo $2^k$.
Hence, for each fixed $i$, the number of pairs with $v_2(x)=i$ is
$2^{k-i-1}\cdot 2^i=2^{k-1}$.
Summing over $i=0,1,\dots,u$ gives
\[
  P_k(t)=(u+1)\,2^{k-1}.
\]

\smallskip
\noindent\emph{Case 2: $t\equiv 0\pmod{2^k}$ (so $u=k$).}
For $0\le i\le k-1$, the same count shows there are $2^{k-i-1}$ residues $x$ with $v_2(x)=i$.
The condition $xy\equiv 0\pmod{2^k}$ then forces $y\equiv 0\pmod{2^{k-i}}$, giving $2^i$ choices for $y\bmod 2^k$.
Thus each $i\in\{0,\dots,k-1\}$ contributes $2^{k-1}$ pairs.
Finally, for $x\equiv 0\pmod{2^k}$ there are $2^k$ choices for $y$.
Therefore
\[
  P_k(0)=k\cdot 2^{k-1}+2^k=(k+2)\,2^{k-1}.
\]

Finally, substitute $k=m-1$ into $M_m(s)=4P_{m-1}(t)$.
Write $v:=v_2(s)$ with the convention $v=m$ if $s\equiv 0\pmod{2^m}$.
If $1\le v<m$, then $s=2t$ implies $v_2(t)=v-1$, so $u=v-1$ and we find
\[
  M_m(s)=4(u+1)2^{(m-1)-1}=v\cdot 2^m.
\]
If $v=m$, then $t\equiv 0\pmod{2^{m-1}}$ so $u=k=m-1$ and
\[
  M_m(0)=4\,(k+2)2^{k-1}=(m+1)\,2^m.
\]
This yields the stated formulas.
\end{proof}

\begin{lemma}[Scaling reduction for the hyperbolic block]\label{lem:scaling-hyp}
Let $p$ be a prime, $m\ge 1$, and $e\ge 0$.
Consider the scaled congruence
\begin{equation}\label{eq:scaled-hyperbolic}
  2p^e xy\equiv s\pmod{p^m}.
\end{equation}
Define
\[
  M^{(e)}_{p,m}(s):=\#\{(x,y)\in(\Z/p^m\Z)^2: 2p^e xy\equiv s\pmod{p^m}\}.
\]
If $v_p(s)<e$ then $M^{(e)}_{p,m}(s)=0$.
If $v_p(s)\ge e$, then writing $s=p^e s_e$ one has
\[
  M^{(e)}_{p,m}(s)=p^{2e}\,M_{p,m-e}(s_e)
\]
for all $m>e$, with the convention $M_{p,0}(\cdot)=1$.
When $p=2$ we abbreviate $M^{(e)}_m(s):=M^{(e)}_{2,m}(s)$.
\end{lemma}

\begin{proof}
If $v_p(s)<e$, then the left-hand side of \eqref{eq:scaled-hyperbolic} is divisible by $p^e$ for every $(x,y)$, so no solutions exist.
Assume $v_p(s)\ge e$ and write $s=p^e s_e$.
Reducing \eqref{eq:scaled-hyperbolic} modulo $p^{m-e}$ gives
\[
  2xy\equiv s_e\pmod{p^{m-e}}.
\]
Any solution $(x_0,y_0)$ modulo $p^{m-e}$ lifts to exactly $p^{2e}$ solutions modulo $p^m$ by choosing arbitrary lifts of $x_0$ and $y_0$ through the additional $e$ $p$-adic digits.
When $m=e$ the reduced modulus is $p^0=1$, so there is a unique residue class, which we record as the convention $M_{p,0}(\cdot)=1$.
\end{proof}

\subsection*{Generating series}

For some applications it is convenient to package the counts into a generating series.
For a prime $p$, an exponent $e\ge 0$, and $s\in\Zp$, define
\begin{equation}\label{eq:hyperbolic-genfun}
  \mathcal{B}^{(e)}_{p,s}(X)
  :=\sum_{m\ge 1} M^{(e)}_{p,m}(s)\,X^m.
\end{equation}
We view $X$ as a formal variable, so \eqref{eq:hyperbolic-genfun} is a formal power series.
Lemma~\ref{lem:scaling-hyp} reduces \eqref{eq:hyperbolic-genfun} to the unscaled case $e=0$, so we record explicit closed forms for $\mathcal{B}^{(0)}_{p,s}$.

\begin{proposition}[Closed form for odd primes]\label{prop:hyperbolic-series-odd}
Let $p$ be an odd prime and let $s\in\Zp$ with $v:=v_p(s)\in\Z_{\ge 0}\cup\{\infty\}$.
Then the series
\[
  \mathcal{B}^{(0)}_{p,s}(X)=\sum_{m\ge 1} M_{p,m}(s)\,X^m
\]
admits the explicit decomposition
\[
  \mathcal{B}^{(0)}_{p,s}(X)
  =
  \underbrace{\sum_{m=1}^{v} p^{m-1}\bigl(p+(p-1)m\bigr)X^m}_{=:P_{p,v}(X)}
  +\underbrace{\frac{(v+1)(p-1)p^v X^{v+1}}{1-pX}}_{=:T_{p,v}(X)},
\]
with the conventions that the finite sum $P_{p,v}$ is empty when $v=0$, and that for $s=0$ (i.e.\ $v=\infty$) the tail term is absent and the first term sums to the convergent closed form
\[
  \mathcal{B}^{(0)}_{p,0}(X)
  =X\left(\frac{p}{1-pX}+\frac{p-1}{(1-pX)^2}\right).
\]
\end{proposition}

\begin{proof}
If $s=0$, then $M_{p,m}(0)=p^{m-1}(p+(p-1)m)$ for all $m\ge 1$ by Lemma~\ref{lem:2xy-odd}(ii), and the displayed rational function is the sum of the resulting elementary $q$-series.
Assume $s\ne 0$ and write $v=v_p(s)<\infty$.
For $m\le v$ one has $s\equiv 0\pmod{p^m}$, hence Lemma~\ref{lem:2xy-odd}(ii) gives $M_{p,m}(s)=M_{p,m}(0)=p^{m-1}(p+(p-1)m)$.
For $m>v$, Lemma~\ref{lem:2xy-odd}(i) gives $M_{p,m}(s)=(v+1)(p-1)p^{m-1}$, which sums to the geometric tail $T_{p,v}(X)$.
\end{proof}

\begin{proposition}[Closed form at $p=2$]\label{prop:hyperbolic-series-two}
Let $s\in\Ztwo$ and $v:=v_2(s)\in\Z_{\ge 0}\cup\{\infty\}$.
Then:
\begin{altenumerate}
\item If $v=0$ (i.e.\ $s$ is odd), then $\mathcal{B}^{(0)}_{2,s}(X)=0$.
\item If $1\le v<\infty$, then
\[
  \mathcal{B}^{(0)}_{2,s}(X)
  =\sum_{m=1}^{v}(m+1)2^{m}X^m
  +\frac{v\cdot 2^{v+1} X^{v+1}}{1-2X}.
\]
\item If $s=0$ (i.e.\ $v=\infty$), then
\[
  \mathcal{B}^{(0)}_{2,0}(X)=\sum_{m\ge 1}(m+1)2^{m}X^m=\frac{4X(1-X)}{(1-2X)^2}.
\]
\end{altenumerate}
\end{proposition}

\begin{proof}
If $v=0$ then $M_m(s)=0$ for all $m$ by Lemma~\ref{lem:2xy-2}(i).
If $1\le v<\infty$, then for $m\le v$ one has $s\equiv 0\pmod{2^m}$ so Lemma~\ref{lem:2xy-2}(iii) gives $M_m(s)=M_m(0)=(m+1)2^m$; and for $m>v$, Lemma~\ref{lem:2xy-2}(ii) gives $M_m(s)=v\cdot 2^m$.
Repackaging these values into generating-series form yields (ii), where the second term is the geometric tail $\sum_{m\ge v+1} v\cdot 2^m X^m$.
The $s=0$ case is the same computation without a tail.
\end{proof}


%% file: sections/04_anisotropic_even_plane.tex
\section{The anisotropic even plane}
\label{sec:anisotropic}

We next treat the anisotropic even plane $H_1$ over $\Ztwo$.
After rewriting the form $x^2+xy+y^2$ as a norm form in the unramified quadratic extension of $\Qtwo$, one obtains a clean recursion for lifting unit solutions.
This yields the stable-range formulas for the counts $A_m(t)$ in Lemma~\ref{lem:H1}, together with the scaling reduction in Lemma~\ref{lem:scaling-H1}.

The anisotropic plane $H_1$ has Gram matrix
$\begin{psmallmatrix}2&1\\1&2\end{psmallmatrix}$.
Its quadratic form is
\[Q(x,y)=2(x^2+xy+y^2).\]

For $m\ge 1$ and $t\in\Ztwo$ define
\[
  A_m(t) := \#\{(x,y)\in(\Z/2^m\Z)^2 : 2(x^2+xy+y^2)\equiv t\pmod{2^m}\}.
\]
This count depends only on the residue class of $t$ modulo $2^m$.

\begin{lemma}[Counts for the anisotropic even plane]\label{lem:H1}
Let $m\ge 1$ and $t\in\Ztwo$.
Write $v:=v_2(t)$ with the convention $v=\infty$ if $t=0$.
Then:
\begin{altenumeratealpha}
\item If $t\ne 0$ and $v$ is even, then $A_m(t)=0$ for all $m>v$.
\item If $t\ne 0$ and $v$ is odd, then $A_m(t)=3\cdot 2^m$ for all $m>v$.
\item If $t\equiv 0\pmod{2^m}$, then $A_m(t)=4^{\lceil m/2\rceil}$.
\end{altenumeratealpha}
\end{lemma}

\begin{lemma}[Scaling reduction for the anisotropic block]\label{lem:scaling-H1}
Let $e\ge 0$ and consider the scaled lattice $2^eH_1$.
Its quadratic form is
\[
  Q_e(x,y)=2^{e+1}(x^2+xy+y^2).
\]
For $m>e$ and $t\in\Ztwo$ set
\[
  A^{(e)}_m(t):=\#\{(x,y)\in(\Z/2^m\Z)^2: Q_e(x,y)\equiv t\pmod{2^m}\}.
\]
If $v_2(t)<e$ then $A^{(e)}_m(t)=0$.
If $v_2(t)\ge e$, writing $t=2^e t_e$ one has
\[
  A^{(e)}_m(t)=2^{2e}\,A_{m-e}(t_e)
\]
where $A_{\bullet}(\cdot)$ is the unscaled count from Lemma~\ref{lem:H1}.
\end{lemma}

\begin{proof}
If $v_2(t)<e$ then $Q_e(x,y)$ is divisible by $2^e$ for all $(x,y)$ and no solutions exist.
If $t=2^e t_e$, reducing the congruence modulo $2^{m-e}$ yields
$2(x^2+xy+y^2)\equiv t_e\pmod{2^{m-e}}$, i.e.\ an instance of Lemma~\ref{lem:H1}.
Each solution modulo $2^{m-e}$ lifts to exactly $2^{2e}$ solutions modulo $2^m$ by independently adding multiples of $2^{m-e}$ to $x$ and $y$.
\end{proof}

\begin{proof}[Proof of Lemma~\ref{lem:H1}]
\noindent\emph{Strategy.} We rewrite $x^2+xy+y^2$ as the norm from $\mathcal O=\Ztwo[\omega]$; we use that $v_2(N(z))$ is always even and that the norm map on units has uniform lifting, and we split according to $v_2(t)$.

Let $\omega$ be a root of $\omega^2+\omega+1=0$ and set $\mathcal O:=\Ztwo[\omega]$, the ring of integers of the unramified quadratic extension $\Qtwo(\omega)/\Qtwo$.
The norm satisfies $N(x+y\omega)=x^2+xy+y^2$.
The congruence $2(x^2+xy+y^2)\equiv t\pmod{2^m}$ is equivalent to
\begin{equation}\label{eq:norm-cong}
  N(z)\equiv t/2\pmod{2^{m-1}},\qquad z\in\mathcal O/2^m\mathcal O.
\end{equation}

Because $N(z)\bmod 2^{m-1}$ depends only on $z\bmod 2^{m-1}$, each solution modulo $2^{m-1}$ has exactly $4$ lifts modulo $2^m$.
Thus the problem reduces to understanding which residue classes occur as $N(z)\bmod 2^{m-1}$, and with what multiplicity.

Every $z\in\mathcal O$ factors uniquely as $z=2^k u$ with $u\in\mathcal O^\times$ and $k\ge 0$.
Since $N(u)\in\Ztwo^\times$ for $u\in\mathcal O^\times$, we have $v_2(N(u))=0$, hence
\[
  v_2(N(z))=2k
\]
is always \emph{even}.

\smallskip
\noindent\emph{(a).}
Assume $t\ne 0$ and $v_2(t)=v$ is even.
Then $v_2(t/2)=v-1$ is odd.
If $m>v$, then $m-1>v_2(t/2)$ and any solution to \eqref{eq:norm-cong} would satisfy
\[
  v_2\bigl(N(z)\bigr)=v_2(t/2)=v-1,
\]
since two $2$-adic integers congruent modulo $2^{m-1}$ have the same valuation whenever that valuation is $<m-1$.
This is impossible because $v_2(N(z))$ is even.
Hence $A_m(t)=0$ for all $m>v$, as claimed.

\smallskip
\noindent\emph{(b).}
Assume $t\ne 0$ with $v_2(t)=2k+1$ odd.
Write $t/2=2^{2k}u_0$ with $u_0$ odd, and set $z=2^k w$.
Then \eqref{eq:norm-cong} is equivalent to
\[
  N(w)\equiv u_0\pmod{2^{m-1-2k}}.
\]
Set $M:=m-1-2k\ge 1$.
Since $u_0$ is odd, the congruence $N(w)\equiv u_0\pmod{2^M}$ forces $w\in(\mathcal O/2^M\mathcal O)^\times$.

Modulo $2$ one has $\mathcal O/2\mathcal O\simeq \mathbb{F}_4$, so $(\mathcal O/2\mathcal O)^\times$ has $3$ elements and its norm map to $(\Z/2\Z)^\times=\{1\}$ is trivial.
In particular, $N(w)\equiv 1\pmod{2}$ holds for exactly $3$ unit classes.

\smallskip
\noindent\emph{Unit lifting.}
For $M\ge 1$, any unit solution modulo $2^M$ lifts to exactly two unit solutions modulo $2^{M+1}$.
Indeed, write a unit $w\equiv x+y\omega\pmod{2^M}$ with $x,y\in\Z/2^M\Z$, and consider lifts
\[
  (x',y')=(x+2^M a,\,y+2^M b),\qquad a,b\in\{0,1\}.
\]
A binomial expansion gives
\[
  N(x',y')\equiv N(x,y) + 2^M\bigl((2x+y)a+(x+2y)b\bigr)\pmod{2^{M+1}},
\]
so if $N(x,y)\equiv u_0\pmod{2^M}$ we may write
\[
  N(x,y)=u_0+2^M\varepsilon\qquad (\varepsilon\in\{0,1\}\text{ determined modulo }2).
\]
The lift $(x',y')$ satisfies $N(x',y')\equiv u_0\pmod{2^{M+1}}$ if and only if
\[
  (2x+y)a+(x+2y)b\equiv -\varepsilon\pmod{2}.
\]
Reducing coefficients modulo $2$ gives the equivalent linear condition
\[
  ya+xb\equiv \varepsilon\pmod{2}.
\]
Conceptually, this is a Hensel-type lifting statement for the norm map on $\mathcal O^\times$.
Since $w$ is a unit, $(x,y)$ are not both even, so this linear equation has exactly two solutions $(a,b)\in(\Z/2\Z)^2$.
Iterating, the number of unit solutions modulo $2^M$ is therefore
\[
  3\cdot 2^{M-1}.
\]

Finally, we return from unit solutions modulo $2^M$ to solutions modulo $2^{m-k}$.
Since $z$ is taken modulo $2^m\mathcal O$, the variable $w=z/2^k$ ranges over $\mathcal O/2^{m-k}\mathcal O$.
The congruence $N(w)\equiv u_0\pmod{2^M}$ depends only on the reduction of $w$ modulo $2^M$.

Note that
\[
  m-k=(m-1-2k)+(k+1)=M+(k+1),
\]
so the reduction map $\mathcal O/2^{m-k}\mathcal O\twoheadrightarrow \mathcal O/2^M\mathcal O$ has kernel
$2^M\mathcal O/2^{m-k}\mathcal O$ of size
\[
  \#(2^M\mathcal O/2^{m-k}\mathcal O)=2^{2((m-k)-M)}=2^{2(k+1)}=4^{k+1}
\]
(using that $\mathcal O$ is free of rank~$2$ over $\Ztwo$).
Therefore each unit solution modulo $2^M$ lifts to exactly $4^{k+1}$ solutions modulo $2^{m-k}$.
Multiplying with the count $3\cdot 2^{M-1}$ gives
\[
  \#\{w\bmod 2^{m-k}: N(w)\equiv u_0\ (2^M)\}=(3\cdot 2^{M-1})\cdot 4^{k+1}=3\cdot 2^m.
\]
By \eqref{eq:norm-cong} this is exactly $A_m(t)$, proving~(b).

\smallskip
\noindent\emph{(c).}
If $t\equiv 0\pmod{2^m}$, then \eqref{eq:norm-cong} becomes $N(z)\equiv 0\pmod{2^{m-1}}$.
Write $z=2^r u$ with $u\in\mathcal O^\times$.
Then $v_2(N(z))=2r$, so the condition $N(z)\equiv 0\pmod{2^{m-1}}$ is equivalent to $2r\ge m-1$, i.e.
\[
  r\ge \left\lceil\frac{m-1}{2}\right\rceil.
\]
Thus $z\equiv 0\pmod{2^{\lceil (m-1)/2\rceil}}$.
The number of such residue classes modulo $2^m\mathcal O$ is
\[
  \#\bigl(2^{\lceil (m-1)/2\rceil}\mathcal O/2^m\mathcal O\bigr)
  =2^{2(m-\lceil (m-1)/2\rceil)}
  =4^{\lceil m/2\rceil},
\]
which is~(c).
\end{proof}

\begin{remark}[On generating series for $H_1$]\label{rem:H1-generating-series}
Fix $t\in\Z$.
If $t\ne 0$, parts~(a)--(b) of Lemma~\ref{lem:H1} show that the sequence $A_m(t)$ stabilises once $m>v_2(t)$.
Consequently the generating series $\sum_{m\ge 1}A_m(t)X^m$ is a finite initial segment plus a geometric tail, hence a rational function of~$X$.
For the singular target $t=0$, part~(c) gives $A_m(0)=4^{\lceil m/2\rceil}$, and one may sum explicitly
\[
  \sum_{m\ge 1}4^{\lceil m/2\rceil}X^m
  =\sum_{k\ge 0}4^{k+1}(X^{2k+1}+X^{2k+2})
  =\frac{4X(1+X)}{1-4X^2}.
\]
We do not emphasise these series further, since later sections use only the stable-range behaviour that feeds into the density.
\end{remark}

\begin{remark}\label{rem:H1-not-hensel}
Lemma~\ref{lem:H1} does \emph{not} contradict the general failure of the naive Jacobian/Hensel criterion for bilinear quadratic forms at $p=2$:
we do not apply Hensel lifting to the quadratic form $2(x^2+xy+y^2)$.
Instead we divide out the uniform factor of $2$ and identify $x^2+xy+y^2$ with a \emph{norm form} on an unramified extension, where the norm map is nondegenerate on units.
The genuinely degenerate case is the ternary block $L_3$, treated next.
\end{remark}


%% file: sections/05_half_lift_and_L3.tex
\section{The half-lift principle and the anisotropic ternary block}
\label{sec:half-lift}

This section contains the main dyadic lifting mechanism used later in the density computations of \S\ref{sec:densities}.
We first prove the half-lift principle (Theorem~\ref{thm:half-lift-principle}) for diagonal sums of squares.
We then specialise to the ternary lattice $L_3$ and combine half-lift with a $4$-adic descent and a base computation modulo $8$ to obtain closed formulas in the stable range.

\subsection{A general half-lift principle}

Theorem~\ref{thm:half-lift-principle} is a dyadic analogue of the usual Hensel growth factor for smooth hypersurfaces.
Fix $d\ge 1$ and consider the quadratic form $Q_d(x)=\sum_i x_i^2$.
For ``primitive'' targets $a$ with $4\nmid a$ and for $n\ge 3$, the theorem asserts that the number of solutions to
\[
  Q_d(x)\equiv a\pmod{2^n}
\]
grows by the factor $2^{d-1}$ when passing from $n$ to $n+1$.
The argument is combinatorial: we pair solutions by translation by $2^{n-1}$ in a suitably chosen odd coordinate, and the translation toggles the congruence modulo~$2^{n+1}$.

We begin with a simple fibre-invariance lemma.

\begin{lemma}[Fibre invariance for sums of squares]\label{lem:fibre-invariance}
Let $d\ge 1$ and $n\ge 1$.
For any $x\in(\Z/2^n\Z)^d$ and any $z\in(\Z/2\Z)^d$ one has
\[
  Q_d(x+2^n z)\equiv Q_d(x)\pmod{2^{n+1}}.
\]
Consequently, for any $a\in\Z$ and any residue class $\bar x\in(\Z/2^n\Z)^d$, the number of lifts $\tilde x\in(\Z/2^{n+1}\Z)^d$ of $\bar x$ satisfying
$Q_d(\tilde x)\equiv a\pmod{2^{n+1}}$
is either $0$ or the full fibre size $2^d$.
\end{lemma}

\begin{proof}
Write $\tilde x=x+2^n z$.
Then
\[
  Q_d(\tilde x)=\sum_i(x_i+2^n z_i)^2
  =\sum_i x_i^2 + 2^{n+1}\sum_i x_i z_i + 2^{2n}\sum_i z_i^2.
\]
Both error terms are divisible by $2^{n+1}$ (since $2^{n+1}\mid 2^{n+1}x_i z_i$ and $2^{n+1}\mid 2^{2n}$ for $n\ge 1$), proving the congruence.
The final assertion follows because the reduction map $(\Z/2^{n+1}\Z)^d\to(\Z/2^n\Z)^d$ has fibres of size $2^d$.
\end{proof}

\begin{proof}[Proof of Theorem~\ref{thm:half-lift-principle}]
\noindent\emph{Strategy.} We compare solutions modulo $2^n$ and $2^{n+1}$ using translations by $2^{n-1}$-torsion: for $4\nmid a$ every solution has an odd coordinate, and the binomial expansion shows that these translations pair solutions so that exactly one element in each orbit lifts.

Fix $d\ge 1$, $n\ge 3$, and $a$ with $4\nmid a$.
On $(\Z/2^n\Z)^d$ set $f(x):=Q_d(x)-a$.
For $u\in(\Z/2\Z)^d$ and $x\in(\Z/2^n\Z)^d$ define the translation
\[
  \tau_u(x):=x+2^{n-1}u\in(\Z/2^n\Z)^d.
\]
A binomial expansion gives, for $n\ge 3$,

\begin{equation}\label{eq:expansion-d}
  f(\tau_u(x))\equiv f(x)\pmod{2^n},\qquad
  f(\tau_u(x))\equiv f(x)+2^n(x\cdot u)\pmod{2^{n+1}},
\end{equation}
where $x\cdot u:=\sum_i x_i u_i$ and the quadratic term $2^{2(n-1)}\sum u_i^2$ vanishes modulo $2^{n+1}$ because $2(n-1)\ge n+1$.

Let
$S_n(a):=\{x\in(\Z/2^n\Z)^d: f(x)\equiv 0\pmod{2^n}\}$,
so that $\#S_n(a)=N_{d,n}(a)$.
By Lemma~\ref{lem:fibre-invariance}, for each $\bar x\in(\Z/2^n\Z)^d$ the congruence condition modulo $2^{n+1}$ is constant across the entire fibre of its lifts.
In particular, if we view $f(\bar x)\bmod 2^{n+1}$ as the residue obtained by evaluating $f$ on any lift of $\bar x$ to $(\Z/2^{n+1}\Z)^d$ (well-defined by the lemma), then
\[
  N_{d,n+1}(a)=2^d\cdot\#\{\bar x\in S_n(a): f(\bar x)\equiv 0\pmod{2^{n+1}}\}.
\]
We show the subset on the right has size exactly $\tfrac12\#S_n(a)$.

\smallskip
\noindent\emph{Step 1: every solution has an odd coordinate.}
If all coordinates of $x$ are even then $Q_d(x)\equiv 0\pmod{4}$.
Since $4\nmid a$, no such $x$ lies in $S_n(a)$.
Thus every $x\in S_n(a)$ has at least one odd coordinate.

\smallskip
\noindent\emph{Step 2: partition by the minimal odd coordinate.}
For $1\le i\le d$ let
\[
  S_n^{(i)}(a):=\{x\in S_n(a): x_1,\dots,x_{i-1}\text{ are even and }x_i\text{ is odd}\}.
\]
Then $S_n(a)$ is the disjoint union of the subsets $S_n^{(i)}(a)$.
Moreover each $S_n^{(i)}(a)$ is stable under the involution $\tau_{e_i}$, where $e_i$ is the $i$th standard basis vector:
adding $2^{n-1}$ to an odd coordinate preserves oddness, and it does not affect earlier coordinates.

\smallskip
\noindent\emph{Step 3: half of each subset lifts.}
Fix $i$ and take $u=e_i$.
For $x\in S_n^{(i)}(a)$ one has $x\cdot u=x_i\equiv 1\pmod{2}$.
By \eqref{eq:expansion-d},
$f(\tau_u(x))\equiv f(x)+2^n\pmod{2^{n+1}}$.
Therefore exactly one of the pair $\{x,\tau_u(x)\}$ satisfies $f\equiv 0\pmod{2^{n+1}}$.
Thus each orbit $\{x,\tau_u(x)\}$ contributes exactly one element lifting to level $2^{n+1}$.
As $\tau_u$ acts freely on $S_n^{(i)}(a)$, exactly half of $S_n^{(i)}(a)$ lifts to level $2^{n+1}$.
Summing over $i$ yields
$\#\{x\in S_n(a): f(x)\equiv 0\pmod{2^{n+1}}\}=\tfrac12\#S_n(a)$.

Combining with the $2^d$-to-$1$ fibre size gives
$N_{d,n+1}(a)=2^d\cdot\tfrac12\,N_{d,n}(a)=2^{d-1}N_{d,n}(a)$.
\end{proof}

\begin{remark}\label{rem:half-lift-hensel}
For $d=3$ one can improve Step~2 by choosing a single involution direction on each parity class (as in the original ``three-squares'' half-lift lemma).
The minimal-index partition above is slightly less symmetric but works uniformly for all $d\ge 1$.
In particular, the growth factor $2^{d-1}$ coincides with the ``smooth'' Hensel factor $p^{d-1}$ specialised to $p=2$, even though the Jacobian of the bilinear quadratic form vanishes modulo $2$.
\end{remark}

\subsection{Specialisation to the ternary block $L_3$}

Let $L_3=\langle 2\rangle^{\oplus 3}$ as in \eqref{eq:L3}.
We combine the half-lift principle (which applies when $4\nmid a$) with a separate $4$-adic descent for imprimitive targets and an initial count modulo~$8$.
For $n\ge 0$ and $a\in\Z$ set
\[
  N_n(a):=N_{3,n}(a)=\#\{x\in(\Z/2^n\Z)^3: x_1^2+x_2^2+x_3^2\equiv a\pmod{2^n}\}.
\]
For $n=0$ we interpret $\Z/2^0\Z=\Z/1\Z$, so $N_0(a)=1$ for all $a$.

\begin{lemma}[Reduction to three squares]\label{lem:reduction}
For $m\ge 1$ and $t\in\Z$,
\[
  r_m(t;L_3)=
  \begin{cases}
    0, & t\text{ odd},\\
    8\,N_{m-1}(t/2), & t\text{ even}.
  \end{cases}
\]
\end{lemma}

\begin{proof}
If $t$ is odd there are no solutions because $Q(x)=2(\cdots)$ is even.
If $t=2s$, then $2(\sum x_i^2)\equiv 2s\pmod{2^m}$ is equivalent to $\sum x_i^2\equiv s\pmod{2^{m-1}}$.
Reducing modulo $2^{m-1}$ forgets the top bit in each coordinate, so each solution modulo $2^{m-1}$ has $2^3=8$ lifts modulo $2^m$.
\end{proof}

\subsubsection*{A $4$-adic descent}

\begin{lemma}[$4$-adic descent for three squares]\label{lem:4adic-descent}
Let $n\ge 3$ and $a\in\Z$ with $4\mid a$.
Then every solution to $x_1^2+x_2^2+x_3^2\equiv a\pmod{2^n}$ has $x_1,x_2,x_3$ even, and
\[
  N_n(a)=8\,N_{n-2}(a/4).
\]
\end{lemma}

\begin{proof}
Reducing modulo $4$ gives $x_1^2+x_2^2+x_3^2\equiv 0\pmod{4}$.
Since an odd square is $1\pmod{4}$ and there are only three variables, this forces each $x_i$ to be even.

Write $x_i=2y_i$.
The map
\[
  (y_1,y_2,y_3)\bmod 2^{n-1}\longmapsto (2y_1,2y_2,2y_3)\bmod 2^n
\]
is a bijection onto the subset of triples with all coordinates even.

Substituting $x_i=2y_i$ turns the congruence into
$4(y_1^2+y_2^2+y_3^2)\equiv a\pmod{2^n}$, hence
\[
  y_1^2+y_2^2+y_3^2\equiv a/4\pmod{2^{n-2}}.
\]
The right-hand congruence depends only on $(y_1,y_2,y_3)\bmod 2^{n-2}$.
Each solution modulo $2^{n-2}$ has exactly $2^3=8$ lifts modulo $2^{n-1}$ (choose the top bit in each coordinate), and multiplying by~$2$ gives a unique solution modulo~$2^n$.
\end{proof}

\begin{remark}[Why the $4$-adic descent is special to $d=3$]\label{rem:4adic-descent-special}
Lemma~\ref{lem:4adic-descent} uses the fact that with only three variables,
$x_1^2+x_2^2+x_3^2\equiv 0\pmod 4$ forces each $x_i$ to be even.
For $d\ge 4$ this fails (for example, $1^2+1^2+1^2+1^2\equiv 0\pmod 4$),
so no analogous forced-evenness descent holds in higher dimension.
This is the reason the closed three-squares formulas below involve a genuine $4$-adic descent step rather than a uniform argument in~$d$.
\end{remark}

\subsubsection*{Half-lift and base computation}

\begin{corollary}[Half-lift for primitive three-squares targets]\label{cor:half-lift-3}
Assume $n\ge 3$ and $4\nmid a$.
Then $N_{n+1}(a)=4\,N_n(a)$.
\end{corollary}

\begin{proof}
Apply Theorem~\ref{thm:half-lift-principle} with $d=3$.
\end{proof}

\begin{lemma}[Three-squares counts modulo $8$]\label{lem:mod8}
For $a\in\Z$, the value of $N_3(a)$ depends only on $a\bmod 8$ and is given by
\[
  N_3(a)=
  \begin{cases}
    32, & a\equiv 0,4\pmod{8},\\
    96, & a\equiv 1,2,5,6\pmod{8},\\
    64, & a\equiv 3\pmod{8},\\
    0,  & a\equiv 7\pmod{8}.
  \end{cases}
\]
\end{lemma}

\begin{proof}
Squares modulo $8$ take values in $\{0,1,4\}$ with multiplicities $2,4,2$ respectively.
Convolving these distributions for three variables yields the stated table.
\end{proof}

\begin{proposition}[Closed form for $N_n(a)$]\label{prop:Nn-closed}
Let $a\in\Z$ be nonzero and write $a=4^k a_0$ with $k\ge 0$ and $4\nmid a_0$.
Then for all $n\ge 2k+3$,
\[
  N_n(a)=8^k\,N_3(a_0)\,4^{n-2k-3}.
\]
\end{proposition}

\begin{proof}
If $4\mid a$ and $n\ge 3$, Lemma~\ref{lem:4adic-descent} gives $N_n(a)=8\,N_{n-2}(a/4)$.
Iterating $k$ times yields $N_n(a)=8^k\,N_{n-2k}(a_0)$ for $n\ge 2k+3$.
Since $4\nmid a_0$, Corollary~\ref{cor:half-lift-3} gives
$N_{n-2k}(a_0)=N_3(a_0)\,4^{n-2k-3}$.
\end{proof}


%% file: sections/06_explicit_density_formulas.tex
\section{Explicit density formulas}
\label{sec:densities}

We now pass from the finite-level representation counts of the previous sections to the local representation densities \eqref{eq:density-def}.

Throughout we assume $t\ne 0$.
For $t=0$ the normalised sequence in \eqref{eq:density-def} need not converge; see Remark~\ref{rem:singular-targets}.

We retain the notation $r_m(t;L)$ and $\alpha_p(t;L)$ from \S\ref{sec:prelim}, and write $v_p$ for the $p$-adic valuation (so $v_2$ in the dyadic subsections).
We begin with the hyperbolic plane (all primes), then treat the dyadic Type~II blocks and the ternary block $L_3$, and finally summarise the formulas in Table~\ref{tab:block-densities}.

Recall that if $\mathrm{rank}(L)=n$, then the normalising factor in \eqref{eq:density-def} is $p^{-m(n-1)}$.
In particular, for rank $2$ lattices one normalises by $p^{-m}$, for rank $3$ by $p^{-2m}$, and in rank $1$ no normalisation occurs.

\begin{remark}[Singular target $t=0$]\label{rem:singular-targets}
We record two basic pathologies at $t=0$.
For the hyperbolic plane $H_0$, Lemma~\ref{lem:2xy-2}(iii) gives $r_m(0;H_0)=(m+1)2^m$, so $2^{-m}r_m(0;H_0)=m+1\to\infty$ and the density diverges.
For the anisotropic even plane $H_1$, Lemma~\ref{lem:H1}(c) gives $r_m(0;H_1)=4^{\lceil m/2\rceil}$, so $2^{-m}r_m(0;H_1)$ oscillates between $1$ and $2$ and the limit does not exist.
Similar phenomena occur for rank-one blocks.
Since the densities at nonzero targets are the quantities appearing in Siegel mass formulas and in many applications, we restrict to $t\ne 0$ throughout.
\end{remark}

\subsection{Hyperbolic plane (all primes)}

Let $p$ be a prime and let $H_0$ denote the rank-$2$ hyperbolic plane over $\Zp$ with quadratic form $Q(x,y)=2xy$.
For $t\ne 0$ one has
\[
  r_m(t;H_0)=M_{p,m}(t),
\]
where the count $M_{p,m}$ was defined in \S\ref{sec:hyperbolic} (and for $p=2$ we retain the shorthand $M_m=M_{2,m}$).

\begin{proposition}[Odd primes]\label{prop:density-hyp-odd}
Let $p$ be an odd prime and let $t\in\Zp\setminus\{0\}$.
Set $v:=v_p(t)$.
Then
\[
  \alpha_p(t;H_0)=(v+1)\,\frac{p-1}{p}.
\]
\end{proposition}

\begin{proof}
Fix $m>v$.
Lemma~\ref{lem:2xy-odd}(i) gives
\(r_m(t;H_0)=M_{p,m}(t)=(v+1)(p-1)p^{m-1}\), hence
\[
  p^{-m}r_m(t;H_0)=(v+1)\,\frac{p-1}{p}
\]
for all $m>v$.
This constant is $\alpha_p(t;H_0)$.
\end{proof}

\begin{corollary}[Scaled hyperbolic plane at odd primes]\label{cor:density-hyp-scaled-odd}
Let $p$ be an odd prime and let $e\ge 0$.
Let $L=p^eH_0$, so that $Q(x,y)=2p^e xy$.
Let $t\in\Zp\setminus\{0\}$ and set $v:=v_p(t)$.
If $v<e$ then $\alpha_p(t;L)=0$.
If $v\ge e$, then
\[
  \alpha_p(t;L)=(v-e+1)(p-1)p^{e-1}.
\]
\end{corollary}

\begin{proof}
If $v<e$ then $t$ is not divisible by $p^e$ while $Q(L)\subseteq p^e\Zp$, so there are no solutions at any level.
Assume $v\ge e$.
Fix $m>\max\{e,v\}$.
Lemma~\ref{lem:scaling-hyp} gives
\(r_m(t;L)=p^{2e}M_{p,m-e}(t/p^e)\).
Since $m-e>v_p(t/p^e)=v-e$, Lemma~\ref{lem:2xy-odd}(i) yields
\(M_{p,m-e}(t/p^e)=(v-e+1)(p-1)p^{m-e-1}\).
Thus for every $m>\max\{e,v\}$ one has
\[
  p^{-m}r_m(t;L)=(v-e+1)(p-1)p^{e-1},
\]
and the density is this constant.
\end{proof}

\begin{proposition}\label{prop:density-hyp}
Let $t\in\Ztwo\setminus\{0\}$ and set $v:=v_2(t)$.
Then $\alpha_2(t;H_0)=v$ if $v\ge 1$, and $\alpha_2(t;H_0)=0$ if $v=0$.
\end{proposition}

\begin{proof}
If $v=0$ then $t$ is odd, while $Q(x,y)=2xy$ is even for every $(x,y)$, hence there are no solutions.
If $v\ge 1$, Lemma~\ref{lem:2xy-2} gives $r_m(t;H_0)=v\cdot 2^m$ for all $m>v$.
Thus
$2^{-m}r_m(t;H_0)=v$ for all $m>v$.
Hence $\alpha_2(t;H_0)=v$.
\end{proof}

\begin{corollary}[Scaled hyperbolic block]\label{cor:density-hyp-scaled}
Let $e\ge 0$ and let $L=2^eH_0$, so that $Q(x,y)=2^{e+1}xy$.
Let $t\in\Ztwo\setminus\{0\}$ and set $v:=v_2(t)$.
If $v<e+1$ then $\alpha_2(t;L)=0$.
If $v\ge e+1$ then
\[
  \alpha_2(t;L)=2^e\,(v-e).
\]
\end{corollary}

\begin{proof}
Fix $m>\max\{e,v\}$.
Lemma~\ref{lem:scaling-hyp} gives
\(r_m(t;L)=2^{2e}M_{m-e}(t/2^e)\).
If $v<e+1$ then $t/2^e$ is odd, hence $M_{m-e}(t/2^e)=0$.
If $v\ge e+1$, then $v_2(t/2^e)=v-e\ge 1$ and $m-e>v-e$.
Lemma~\ref{lem:2xy-2} yields $M_{m-e}(t/2^e)=(v-e)\,2^{m-e}$, so
\[
  2^{-m}r_m(t;L)=2^{-m}\cdot 2^{2e}\cdot (v-e)\,2^{m-e}=2^e(v-e),
\]
independent of $m$.
Hence $\alpha_2(t;L)=2^e(v-e)$.
\end{proof}

\begin{proof}[Proof of Theorem~\ref{thm:prime-uniform-h0}]
\noindent\emph{Strategy.} We compare the $Q$- and $q=Q/2$-normalisations via Lemma~\ref{lem:q-dictionary} and reduce to the explicit hyperbolic-plane density formulas already computed in Corollaries~\ref{cor:density-hyp-scaled-odd} and~\ref{cor:density-hyp-scaled}.

If $p$ is odd, then $2\in\Zp^\times$ and the map $t\mapsto 2t$ preserves $v_p(t)$.
The congruence $p^e xy\equiv t\pmod{p^m}$ is equivalent to $2p^e xy\equiv 2t\pmod{p^m}$, so
\(\alpha_p^{(q)}(t;p^eH_0)=\alpha_p(2t;p^eH_0)\).
The stated formula is exactly Corollary~\ref{cor:density-hyp-scaled-odd}.

For $p=2$, the lattice $2^eH_0$ is even and Lemma~\ref{lem:q-dictionary} gives
\[
  \alpha_2^{(q)}(t;2^eH_0)=\tfrac12\alpha_2(2t;2^eH_0).
\]

Write $v:=v_2(t)$, so that $v_2(2t)=v+1$.
Corollary~\ref{cor:density-hyp-scaled} gives $\alpha_2(2t;2^eH_0)=0$ when $v<e$.

Assume $v\ge e$.
Then $v+1\ge e+1$ and the same corollary yields
\(\alpha_2(2t;2^eH_0)=2^e\bigl((v+1)-e\bigr)=2^e(v-e+1)\).
Dividing by $2$ we obtain
\[
  \alpha_2^{(q)}(t;2^eH_0)=
  \begin{cases}
    0, & v<e,\\
    2^{e-1}(v-e+1), & v\ge e,
  \end{cases}
\]
which is exactly $(v-e+1)(p-1)p^{e-1}$ with $p=2$.
In particular, the vanishing threshold $v<e$ matches the odd-prime case.
\end{proof}

\subsection{Type II anisotropic block}

\begin{proposition}\label{prop:density-H1}
Let $t\in\Ztwo\setminus\{0\}$.
If $v_2(t)$ is even then $\alpha_2(t;H_1)=0$.
If $v_2(t)$ is odd then $\alpha_2(t;H_1)=3$.
\end{proposition}

\begin{proof}
For $m>v_2(t)$, Lemma~\ref{lem:H1} gives either $A_m(t)=0$ (even valuation) or $A_m(t)=3\cdot 2^m$ (odd valuation).
Thus for every $m>v_2(t)$ one has
\[
  2^{-m}r_m(t;H_1)=2^{-m}A_m(t)=
  \begin{cases}
    0, & v_2(t)\text{ is even},\\
    3, & v_2(t)\text{ is odd}.
  \end{cases}
\]
\end{proof}

\begin{corollary}[Scaled anisotropic block]\label{cor:density-H1-scaled}
Let $e\ge 0$ and let $L=2^eH_1$, so that $Q(x,y)=2^{e+1}(x^2+xy+y^2)$.
Let $t\in\Ztwo\setminus\{0\}$ and set $v:=v_2(t)$.
If $v<e+1$ then $\alpha_2(t;L)=0$.
If $v\ge e+1$, then
\[
  \alpha_2(t;L)=
  \begin{cases}
    3\cdot 2^e, & v-e\text{ is odd},\\
    0, & v-e\text{ is even}.
  \end{cases}
\]
\end{corollary}

\begin{proof}
Fix $m>\max\{e,v\}$.
Lemma~\ref{lem:scaling-H1} gives
\(r_m(t;L)=2^{2e}A_{m-e}(t/2^e)\).
If $v<e+1$ then $t/2^e$ is odd and $A_{m-e}(t/2^e)=0$ by Lemma~\ref{lem:H1}(a).
If $v\ge e+1$, then $v_2(t/2^e)=v-e\ge 1$ and $m-e>v-e$.
Lemma~\ref{lem:H1}(a)--(b) gives $A_{m-e}(t/2^e)=0$ for even $v-e$ and
$A_{m-e}(t/2^e)=3\cdot 2^{m-e}$ for odd $v-e$.
In the latter case,
\[
  2^{-m}r_m(t;L)=2^{-m}\cdot 2^{2e}\cdot 3\cdot 2^{m-e}=3\cdot 2^e,
\]
independent of $m$.
\end{proof}

\subsection{Type I (rank one) blocks}

We also record the density for a rank-$1$ block $L=\langle 2^a u\rangle$.
Since $\mathrm{rank}(L)=1$, the density $\alpha_2(t;L)$ is the eventual constant value of $r_m(t;L)$.

\begin{proposition}\label{prop:density-typeI}
Let $L=\langle 2^a u\rangle$ with $a\ge 0$ and $u\in\Ztwo^\times$, and let $t\in\Ztwo\setminus\{0\}$.
Write $v:=v_2(t)$.
\begin{altenumeratealpha}
\item If $v<a$ or $v-a$ is odd, then $\alpha_2(t;L)=0$.
\item If $v\ge a$ and $v-a=2j$ is even, write $t=2^{a+2j}c$ with $c\in\Ztwo^\times$.
Then
\[
  \alpha_2(t;L)=
  \begin{cases}
    2^{a+j+2}, & u^{-1}c\equiv 1\pmod 8,\\
    0, & u^{-1}c\not\equiv 1\pmod 8.
  \end{cases}
\]
\end{altenumeratealpha}
\end{proposition}

\begin{proof}
\noindent\emph{Strategy.} Reduce $2^a u x^2\equiv t\pmod{2^m}$ to a square-root congruence modulo $2^{m-a}$ and apply Lemma~\ref{lem:square-roots}.

If $v<a$, then $t\notin 2^a\Ztwo$ while $Q(L)\subseteq 2^a\Ztwo$, so there are no solutions at any level and $\alpha_2(t;L)=0$.

Assume $v\ge a$ and write $t=2^a t_a$.
Fix $m>v+2$ (so in particular $m>a$).
Corollary~\ref{cor:typeI-counts} gives
\[
  r_m(t;L)=2^a\cdot\#\{x\bmod 2^{m-a}: x^2\equiv u^{-1}t_a\!\!\pmod{2^{m-a}}\}.
\]
Write $v_2(t_a)=v-a$.

If $v-a$ is odd, then $u^{-1}t_a$ has odd $2$-adic valuation, whereas a square has even valuation.
As $m-a>v-a$, the displayed congruence forces $v_2(x^2)=v_2(u^{-1}t_a)$, a contradiction.
Hence the displayed set is empty and $\alpha_2(t;L)=0$.

If $v-a=2j$ is even, write $t_a=2^{2j}c$ with $c\in\Ztwo^\times$.
As $m-a>2j+2$, Lemma~\ref{lem:square-roots}(ii)(c) shows that the number of solutions modulo $2^{m-a}$ is
$2^{j+2}$ if $u^{-1}c\equiv 1\pmod 8$ and $0$ otherwise.
In the nonzero case we obtain
\[
  r_m(t;L)=2^a\cdot 2^{j+2}=2^{a+j+2}
\]
for all $m>v+2$.
Since $\mathrm{rank}(L)=1$, the density equals this eventual constant.
\end{proof}

\subsection{The anisotropic ternary block $L_3$}

\begin{proposition}\label{prop:density-L3}
Let $t\in\Ztwo\setminus\{0\}$ be even.
Write
\[\frac{t}{2}=4^k a_0\qquad (k\ge 0,\;4\nmid a_0).\]
Then the density for $L_3=\langle 2\rangle^{\oplus 3}$ is
\[
  \alpha_2(t;L_3)=2^{-k-5}\,N_3(a_0),
\]
where $N_3(a_0)$ is given explicitly by Lemma~\ref{lem:mod8}.
If $t$ is odd then $\alpha_2(t;L_3)=0$.
\end{proposition}

\begin{proof}
If $t$ is odd there are no solutions, since $Q(L_3)$ is always even.
Assume $t$ is even.
Write $t/2=4^k a_0$ with $4\nmid a_0$.
Then $t=2^{2k+1}a_0$, and hence $v_2(t)=2k+1+v_2(a_0)\in\{2k+1,2k+2\}$.
Fix $m>v_2(t)+2$.
Then $m-1\ge 2k+3$, so Proposition~\ref{prop:Nn-closed} applies to $N_{m-1}(t/2)$.
Together with Lemma~\ref{lem:reduction} this gives
\[
  r_m(t;L_3)=8\,N_{m-1}(t/2)=8\cdot 8^k N_3(a_0)4^{m-1-2k-3}.
\]
Normalising by $4^m$ (since $\mathrm{rank}(L_3)=3$) and simplifying makes the stabilisation explicit:
\[
  \frac{r_m(t;L_3)}{4^m}
  =8\cdot 8^k N_3(a_0)\,4^{m-2k-4}\cdot 4^{-m}
  =8\cdot 8^k N_3(a_0)\,4^{-2k-4}
  =2^{-k-5}N_3(a_0),
\]
independent of $m$.
\end{proof}

\begin{remark}\label{rem:L3-not-jordan}
The lattice $L_3=\langle 2\rangle^{\oplus 3}$ is included as a distinguished anisotropic ternary $\Ztwo$-lattice (equivalently, $L_3\cong 2\langle 1\rangle^{\oplus 3}$).
It is not one of the basic dyadic blocks of Definition~\ref{def:dyadic-base-blocks}, and we do not attempt to treat arbitrary odd unimodular constituents of rank~$3$ in the Jordan decomposition of Proposition~\ref{prop:jordan}.
We record the density only for $L_3$ itself because it encodes the classical three-squares congruence counts.

Of course, $L_3$ is also an orthogonal sum of three rank-$1$ Type~I blocks.
In principle one could recover $\alpha_2(t;L_3)$ by iterating the finite-level convolution of Proposition~\ref{prop:orthogonal-convolution} together with the rank-$1$ formulas of Appendix~\ref{app:typeI}; this is possible but leads to a longer valuation/residue case analysis.
Section~\ref{sec:half-lift} gives a direct computation via half-lift and $4$-adic descent, and it can be read as a nontrivial test case for the ``compose blocks'' philosophy discussed below.
\end{remark}

\subsection*{Summary table of the computed densities}

For quick reference we collect the final density values for the basic blocks treated in this paper.
The hyperbolic plane $H_0$ is included at all primes; the remaining entries concern the dyadic blocks at $p=2$.
Scaled blocks $2^eH_0$ and $2^eH_1$ are obtained from
Corollaries~\ref{cor:density-hyp-scaled} and~\ref{cor:density-H1-scaled}.

\begin{proposition}[Stable-range thresholds]\label{prop:stable-thresholds}
Let $L$ be one of the lattices treated in this paper, and let $t\in\Ztwo\setminus\{0\}$ with $v:=v_2(t)$.
Then the closed forms for $r_m(t;L)$ used in the proofs of
Propositions~\ref{prop:density-hyp}, \ref{prop:density-H1}, \ref{prop:density-typeI}, and \ref{prop:density-L3}
already hold once $m$ exceeds an explicit bound depending only on $v$:
\begin{altenumerate}
\item For the even planes $H_0$ and $H_1$ (and their scalings), the stabilised formulas hold for all $m>v$.
\item For a Type~I block $\langle 2^a u\rangle$ and for $L_3=\langle 2\rangle^{\oplus 3}$, the stabilised formulas hold for all $m>v+2$.
\end{altenumerate}
In particular, for each fixed $(L,t)$ the density $\alpha_2(t;L)$ is obtained by evaluating the stabilised expression at any $m$ above the corresponding threshold.
\end{proposition}

\begin{proof}
\smallskip
\noindent\emph{Even planes.}
For $H_0$ and $H_1$ this is exactly the point at which the valuation-stratified counts (respectively the norm-recursion) become uniform; see Lemmas~\ref{lem:2xy-2} and~\ref{lem:H1}.

\smallskip
\noindent\emph{Type~I blocks.}
For a Type~I block, one reduces to a square-root congruence modulo $2^{m-a}$; once $m-a>v-a+2$ (i.e.\ $m>v+2$) the unit square-root count stabilises (Lemma~\ref{lem:square-roots}).

\smallskip
\noindent\emph{The lattice $L_3$.}
For $L_3$, write $t/2=4^k a_0$ with $4\nmid a_0$.
Proposition~\ref{prop:Nn-closed} gives a stable closed form for $N_n(t/2)$ once $n\ge 2k+3$.
Since $r_m(t;L_3)=8N_{m-1}(t/2)$ (Lemma~\ref{lem:reduction}), it suffices to ensure $m-1\ge 2k+3$.
The uniform bound $m>v+2$ does this because $v=v_2(t)\ge 2k+1$.
If $a_0$ is even (so $v=2k+2$) this bound is off by one; we prefer the slightly stronger uniform threshold stated in the proposition.
\end{proof}

\begin{proposition}[Orthogonal-sum convolution]\label{prop:orthogonal-convolution}
Let $L$ and $M$ be quadratic $\Zp$-lattices, and equip $L\oplus M$ with the orthogonal direct-sum quadratic form $Q_{L\oplus M}(v,w):=Q_L(v)+Q_M(w)$.
Then for every $m\ge 1$ and $t\in\Zp$ one has the finite-level identity
\[
  r_m(t;L\oplus M)
  \,=\,
  \sum_{s\in \Z/p^m\Z} r_m(s;L)\,r_m(t-s;M).
\]
\end{proposition}

\begin{proof}
An element of $(L\oplus M)/p^m(L\oplus M)$ is a pair $(v,w)$ with $v\in L/p^mL$ and $w\in M/p^mM$.
The congruence $Q_L(v)+Q_M(w)\equiv t\pmod{p^m}$ holds if and only if $Q_L(v)\equiv s\pmod{p^m}$ and $Q_M(w)\equiv t-s\pmod{p^m}$ for some residue class $s\in\Z/p^m\Z$.
Summing over $s$ gives the stated identity.
\end{proof}

\noindent
Proposition~\ref{prop:orthogonal-convolution} shows that once the basic block counts are known, representation counts for orthogonal sums reduce to explicit (finite) convolutions.

At the density level one normalises by $2^{m(n-1)}$ and lets $m\to\infty$.
Although the convolution sum in Proposition~\ref{prop:orthogonal-convolution} ranges over $2^m$ residue classes, for the block families computed in this paper the sum can be organised into finitely many valuation/residue strata whose contributions stabilise.
This is already visible in Example~\ref{ex:H0-plus-1}; we record the general mechanism in a short lemma.

\begin{lemma}[Stable stratification of dyadic block convolutions]\label{lem:stable-convolution}
Let $L,M$ be (scaled) dyadic blocks appearing in Table~\ref{tab:block-densities}.
Fix $t\in\Ztwo\setminus\{0\}$ and set $n:=\mathrm{rank}(L)+\mathrm{rank}(M)$.
Then for all $m$ above the stable thresholds of Proposition~\ref{prop:stable-thresholds} for the relevant block types (for example, for $m>v_2(t)+2$), the normalised convolution sum
\[
  2^{-m(n-1)}\,r_m(t;L\oplus M)
  \,=\,
  2^{-m(n-1)}\sum_{s\bmod 2^m} r_m(s;L)\,r_m(t-s;M)
\]
reduces to a finite sum over explicit valuation/residue strata (indexed by $v_2(s)$ and, when relevant, the residue class of $2^{-v_2(s)}s$ modulo $8$), together with at most a geometric tail indexed by $v_2(t-s)$.
In particular, the density $\alpha_2(t;L\oplus M)$ exists and can be assembled from finitely many stable stratum contributions (equivalently, finitely many evaluations of geometric series), using the explicit block data in Table~\ref{tab:block-densities}.
\end{lemma}

\begin{proof}
Write $r:=\mathrm{rank}(L)$ and $r':=\mathrm{rank}(M)$, so that $n=r+r'$.
For each block $B\in\{L,M\}$ and each fixed residue class $u\bmod 2^m$, the stable-range formulas in \S\ref{sec:densities} show that, once $m$ lies above the stable threshold for $B$, the normalised quantity
\[
  \beta_B(u):=2^{-m(\mathrm{rank}(B)-1)}\,r_m(u;B)
\]
is independent of $m$ and depends on $u$ only through $v_2(u)$ and (when $B$ is Type~I or $L_3$) a bounded amount of unit data, which may be taken as the class of $2^{-v_2(u)}u$ modulo~$8$.
In particular, for all such $m$ we may rewrite the normalised convolution as
\[
  2^{-m(n-1)}r_m(t;L\oplus M)
  =\sum_{s\bmod 2^m} 2^{-m}\,\beta_L(s)\,\beta_M(t-s).
\]

We now stratify $s\bmod 2^m$ so that both factors $\beta_L(s)$ and $\beta_M(t-s)$ are constant on each stratum.
For fixed nonzero $t$, set $w:=v_2(t)$.
The relevant data for $\beta_L(s)$ is $(v_2(s),\,2^{-v_2(s)}s\bmod 8)$ (the second component only when needed), and similarly for $\beta_M(t-s)$.
This already yields only finitely many strata, since $v_2(s)$ and $v_2(t-s)$ are truncated by $m$ and the unit data lives in $(\Z/8\Z)^\times$.

The only source of an \emph{unbounded} family of strata (as $m\to\infty$) is the possibility that $v_2(t-s)$ is much larger than $w$.
But $v_2(t-s)>w$ holds if and only if $s\equiv t\pmod{2^{w+1}}$.
In this tail regime we may write
\[
  s = t + 2^{w+1}u\qquad (u\bmod 2^{m-w-1}),
\]
and then
\[
  v_2(t-s)=v_2(2^{w+1}u)=w+1+v_2(u).
\]
Thus the tail can be indexed by the valuation $v_2(u)=0,1,\dots,m-w-2$ (together with the single class $u\equiv 0\pmod{2^{m-w-1}}$), and the number of such $u$ with $v_2(u)=j$ is $2^{m-w-2-j}$.
Since $\beta_L(s)$ and $\beta_M(t-s)$ are constant on these valuation/residue classes, the tail contribution becomes a finite geometric series in~$j$.
All remaining strata are finite in number and have cardinality of the form $2^{m-\ell}$ with $\ell$ independent of $m$, so their contributions are $m$-independent after the normalising factor $2^{-m}$.
\end{proof}

Deriving compact closed forms for orthogonal sums can still become combinatorial, but the lemma shows that in the present block families one is convolving \emph{stable} pieces, so the density limit may be computed systematically.
The following worked example spells this out in a simple dyadic case.

\begin{example}[A dyadic Jordan sum: $H_0\oplus\langle 1\rangle$]\label{ex:H0-plus-1}
Let $p=2$ and consider the rank-$3$ lattice $L:=H_0\oplus\langle 1\rangle$ with quadratic form
\[
  Q(x,y,z)=2xy+z^2.
\]
Fix an \emph{odd} target $t\in\Ztwo^\times$ and take $m\ge 3$ (so in particular $m>v_2(t-s)+2$ for every contributing residue class $s$).
For each residue class $s\in\Z/2^m\Z$ we have the convolution identity
\[
  r_m(t;L)=\sum_{s\bmod 2^m} r_m(s;H_0)\,r_m(t-s;\langle 1\rangle).
\]

By Lemma~\ref{lem:2xy-2}, the factor $r_m(s;H_0)=M_m(s)$ is nonzero only when $s$ is even.
Since $t$ is odd, any contributing residue class $s$ forces $t-s$ to be odd.
Lemma~\ref{lem:square-roots}(ii) then shows that $r_m(t-s;\langle 1\rangle)$ equals $4$ if $t-s\equiv 1\pmod 8$ and $0$ otherwise.
Thus only the congruence class $s\equiv t-1\pmod 8$ contributes, and
\[
  r_m(t;L)=4\sum_{\substack{s\bmod 2^m\\ s\equiv t-1\pmod 8}} M_m(s).
\]

Write $t\bmod 8$.
There are three cases.

\begin{altitemize}
\item If $t\equiv 3,7\pmod 8$, then any contributing $s\equiv t-1\pmod 8$ has $v_2(s)=1$, hence $M_m(s)=2^m$.
There are $2^{m-3}$ such residue classes, so
\[
  r_m(t;L)=4\cdot 2^{m-3}\cdot 2^m=2^{2m-1}.
\]

\item If $t\equiv 5\pmod 8$, then any contributing $s\equiv 4\pmod 8$ has $v_2(s)=2$, hence $M_m(s)=2\cdot 2^m$.
Thus
\[
  r_m(t;L)=4\cdot 2^{m-3}\cdot (2\cdot 2^m)=2^{2m}.
\]

\item If $t\equiv 1\pmod 8$, then $s\equiv 0\pmod 8$.
Write $s=8u$ with $u\bmod 2^{m-3}$.
Using Lemma~\ref{lem:2xy-2} and the identity
\(
  \sum_{u\bmod 2^{k}} v_2(u)=2^{k}-1
\)
(with the convention $v_2(0)=k$), we compute
\begin{align*}
  \sum_{\substack{s\bmod 2^m\\ s\equiv 0\pmod 8}} M_m(s)
  &=2^m\Bigl(\sum_{u\bmod 2^{m-3}} (3+v_2(u))+1\Bigr)\\
  &=2^m\Bigl(3\cdot 2^{m-3}+(2^{m-3}-1)+1\Bigr)
   =2^{2m-1}.
\end{align*}
Hence $r_m(t;L)=4\cdot 2^{2m-1}=2^{2m+1}$.
\end{altitemize}

In every case the normalised density is obtained by dividing by $2^{m(3-1)}=2^{2m}$:
\[
  \alpha_2(t;H_0\oplus\langle 1\rangle)=
  \begin{cases}
    2, & t\equiv 1\pmod 8,\\
    1, & t\equiv 5\pmod 8,\\
    \tfrac12, & t\equiv 3,7\pmod 8.
  \end{cases}
\]
This illustrates two general features: once the basic block counts are explicit, assembling dyadic Jordan sums is a finite calculation, and the normalising factor $2^{-m(n-1)}$ in \eqref{eq:density-def} compensates precisely for the $2^m$ choices of the intermediate residue class $s$.
\end{example}

\begin{table}[ht]
\centering
\small
{\renewcommand{\arraystretch}{1.15}%
\begin{tabularx}{\textwidth}{@{}>{\raggedright\arraybackslash}p{0.38\textwidth}>{\raggedright\arraybackslash}X@{}}
\toprule
\textbf{Local lattice $L$} & \textbf{$Q$-normalised density $\alpha_p(t;L)$ for $t\in\Zp\setminus\{0\}$}\\
\midrule
\multicolumn{2}{@{}l@{}}{\textbf{Odd primes ($p\ne 2$)}}\\
\addlinespace[0.25em]
$H_0$ over $\Zp$ (odd $p$) &
If $v:=v_p(t)$, then $\alpha_p(t;H_0)=(v+1)\dfrac{p-1}{p}$ (Proposition~\ref{prop:density-hyp-odd}).\\

\addlinespace[0.4em]
\multicolumn{2}{@{}l@{}}{\textbf{Dyadic prime ($p=2$)}}\\
\addlinespace[0.25em]

$H_0$ over $\Ztwo$ &
If $v:=v_2(t)$, then $\alpha_2(t;H_0)=0$ for $v=0$ and $\alpha_2(t;H_0)=v$ for $v\ge 1$ (Proposition~\ref{prop:density-hyp}).\\

\addlinespace[0.45em]

$H_1$ (anisotropic even plane) &
If $v:=v_2(t)$, then $\alpha_2(t;H_1)=0$ for even $v$ and $\alpha_2(t;H_1)=3$ for odd $v$ (Proposition~\ref{prop:density-H1}).\\

\addlinespace[0.45em]

Type~I block $\langle 2^a u\rangle$ ($a\ge 0$, $u\in\Ztwo^\times$) &
\begin{minipage}[t]{\linewidth}
Write $v:=v_2(t)$.
If $v<a$ or $v-a$ is odd, then $\alpha_2(t;\langle 2^a u\rangle)=0$.
Otherwise write $t=2^{a+2j}c$ with $c\in\Ztwo^\times$.
Then
\[
  \alpha_2(t;\langle 2^a u\rangle)=
  \begin{cases}
    2^{a+j+2}, & u^{-1}c\equiv 1\pmod 8,\\
    0, & u^{-1}c\not\equiv 1\pmod 8.
  \end{cases}
\]
(Proposition~\ref{prop:density-typeI}).
\end{minipage}\\

\addlinespace[0.6em]

$L_3=\langle 2\rangle^{\oplus 3}$ &
\begin{minipage}[t]{\linewidth}
If $t$ is odd, then $\alpha_2(t;L_3)=0$.
If $t$ is even, write $t/2=4^k a_0$ with $k\ge 0$ and $4\nmid a_0$ (so $a_0\not\equiv 0,4\pmod 8$).
Then
\[
  \alpha_2(t;L_3)=2^{-k}\cdot
  \begin{cases}
    3, & a_0\equiv 1,2,5,6\pmod 8,\\
    2, & a_0\equiv 3\pmod 8,\\
    0, & a_0\equiv 7\pmod 8.
  \end{cases}
\]
(Proposition~\ref{prop:density-L3} and Lemma~\ref{lem:mod8}).
\end{minipage}\\
\bottomrule
\end{tabularx}}
\caption{Local representation densities (in the $Q(v)=\langle v,v\rangle$ normalisation of \eqref{eq:density-def}) for the basic blocks computed in this paper.
The hyperbolic plane $H_0$ is included at all primes; the remaining rows concern $p=2$.
For even dyadic blocks one may instead work with $q:=Q/2$; Lemma~\ref{lem:q-dictionary} gives the exact conversion $\alpha_2(2t;L)=2\,\alpha_2^{(q)}(t;L)$.
Under the $q$-normalisation, the two $H_0$ rows combine into the prime-uniform formula of Theorem~\ref{thm:prime-uniform-h0}.}
\label{tab:block-densities}
\end{table}

\begin{corollary}[Dyadic vanishing thresholds in the $Q$- and $q$-normalisations]
\label{cor:dyadic-vanishing-thresholds}
Let $p=2$.
Let $L$ be an \emph{even} dyadic block appearing in Table~\ref{tab:block-densities}
(so $L$ is one of $2^eH_0$, $2^eH_1$, $L_3$, or $\langle 2^a u\rangle$ with $a\ge 1$).
For $t\in\Ztwo\setminus\{0\}$ write $v:=v_2(t)$.

\begin{altenumeratealpha}
\item \emph{$Q$-normalisation.}
The density $\alpha_2(t;L)$ vanishes unless the following valuation constraints hold:
\begin{altitemize}
\item $L=2^eH_0$: one must have $v\ge e+1$ (Corollary~\ref{cor:density-hyp-scaled}).
\item $L=2^eH_1$: one must have $v\ge e+1$ and $v-e$ odd (Corollary~\ref{cor:density-H1-scaled}).
\item $L=L_3$: one must have $v\ge 1$ (Lemma~\ref{lem:reduction}); further unit obstructions are recorded in Proposition~\ref{prop:density-L3}.
\item $L=\langle 2^a u\rangle$ with $a\ge 1$: one must have $v\ge a$ and $v-a$ even (Proposition~\ref{prop:density-typeI}).
\end{altitemize}

\item \emph{$q=Q/2$-normalisation.}
For even $L$ write $q:=Q/2$ and $t=2t'$.
Lemma~\ref{lem:q-dictionary} gives the exact conversion
\(\alpha_2(2t';L)=2\,\alpha_2^{(q)}(t';L)\).
In particular, the systematic valuation constraints in~(a) shift down by one when passing from $(Q,t)$ to $(q,t')$:
for example $v_2(t)\ge e+1$ for $2^eH_0$ becomes $v_2(t')\ge e$, and $v_2(t)\ge e+1$ with $v_2(t)-e$ odd for $2^eH_1$ becomes $v_2(t')\ge e$ with $v_2(t')-e$ even.
\end{altenumeratealpha}
\end{corollary}

\subsection*{Normalisation and comparison with Siegel-series formulas}

The dyadic densities computed here agree with the general formulas for the dyadic Siegel series (see for example
\cite{grossKeating1993,ikedaKatsurada2018,choYamauchi2020,yang1998}).
For instance, Yang~\cite{yang1998} gives an explicit formula for local densities in terms of local invariants.
Those formulas are usually stated for the quadratic form $q(v)=\langle v,v\rangle/2$ on even lattices; in this paper we use the bilinear-lattice convention $Q(v)=\langle v,v\rangle$, and Lemma~\ref{lem:q-dictionary} gives the exact conversion.

To make the comparison completely concrete, many references (including Yang~\cite{yang1998}) use the $q$-normalisation on even dyadic lattices and fix the Haar measure by requiring $\mathrm{vol}(L)=1$.
With this normalisation, the standard dyadic local density is computed by the same limiting procedure as in \eqref{eq:density-def}, but with $q$ in place of $Q$; in our notation this is $\alpha_2^{(q)}(t;L)$.
Thus Table~\ref{tab:block-densities} can be compared directly to the classical formulas by the identity
\[
  \alpha_2(2t;L)=2\,\alpha_2^{(q)}(t;L)\qquad (L\ \text{even}),
\]
from Lemma~\ref{lem:q-dictionary}.

As a fixed anchor, Proposition~\ref{prop:density-H1} gives $\alpha_2(2;H_1)=3$.
In the $q$-normalisation this becomes
\[
  \alpha_2^{(q)}(1;H_1)=\tfrac32,
\]
which is the quantity that should be compared to the dyadic Siegel-series formulas for the even unimodular anisotropic plane.

In the block cases treated here, the relevant local invariants reduce to valuation data and (for Type~I blocks and $L_3$) a residue class modulo~$8$.
Accordingly, the general Siegel-series formulas specialise to the piecewise constants recorded in
Propositions~\ref{prop:density-hyp}, \ref{prop:density-H1}, \ref{prop:density-typeI}, \ref{prop:density-L3} and Table~\ref{tab:block-densities}.
We mention this only as an explicit bridge to the classical theory; none of the arguments in this paper uses the Siegel-series machinery.


%% file: sections/A_typeI_rank1_blocks.tex
\section{Type I (rank one) dyadic blocks}
\label{app:typeI}

We record the square-root counts modulo $2^m$ that enter the rank-$1$ dyadic blocks.
These counts are classical, but we include them so that the evaluation of the base blocks can be read without appealing to external references.

\begin{lemma}[Square-root counts modulo $2^m$]\label{lem:square-roots}
Let $m\ge 1$ and $b\in\Z/2^m\Z$.
\begin{altenumerate}
\item If $b\equiv 0\pmod{2^m}$, then the number of solutions to $x^2\equiv b\pmod{2^m}$ is $2^{\lfloor m/2\rfloor}$.
\item Otherwise write $b=2^{2j}c$ with $j\ge 0$, $2j<m$, and $c$ odd, and set $k:=m-2j\ge 1$.
Then $x^2\equiv b\pmod{2^m}$ has solutions exactly in the following cases:
\begin{altenumeratealpha}
\item $k=1$: always, and then there are exactly $2^{j}$ solutions;
\item $k=2$: iff $c\equiv 1\pmod{4}$, and then there are exactly $2^{j+1}$ solutions;
\item $k\ge 3$: iff $c\equiv 1\pmod{8}$, and then there are exactly $2^{j+2}$ solutions.
\end{altenumeratealpha}
\end{altenumerate}
\end{lemma}

\begin{proof}
\noindent\emph{Strategy.} We first handle the singular target $b\equiv 0$ by a valuation condition, and then reduce the general congruence $x^2\equiv 2^{2j}c$ to a unit congruence $y^2\equiv c$ by factoring out the maximal square power of~$2$.

\smallskip
\noindent\emph{(i)} The congruence $x^2\equiv 0\pmod{2^m}$ is equivalent to $v_2(x)\ge \lceil m/2\rceil$.
Thus $x$ is determined by its residue class modulo $2^m$ subject to this valuation condition, and the number of such classes equals
\[
  2^{m-\lceil m/2\rceil}=2^{\lfloor m/2\rfloor}.
\]

\smallskip
\noindent\emph{(ii)} Assume $b\not\equiv 0\pmod{2^m}$ and write $b=2^{2j}c$ with $j\ge 0$, $2j<m$, and $c$ odd.
If $x^2\equiv 2^{2j}c\pmod{2^m}$, then necessarily $2^j\mid x$.
Write $x=2^j y$.
Dividing by $2^{2j}$ reduces the congruence to
\begin{equation}\label{eq:sqroot-reduced}
  y^2\equiv c\pmod{2^{k}},\qquad k:=m-2j\ge 1.
\end{equation}
The condition \eqref{eq:sqroot-reduced} depends only on the class of $y$ modulo $2^k$.
Conversely, each residue class $\bar y\in\Z/2^k\Z$ has exactly $2^j$ lifts to a class modulo $2^{m-j}$.
Hence the number of solutions $x\bmod 2^m$ equals
\[
  2^j\cdot\#\{y\bmod 2^k: y^2\equiv c\pmod{2^k}\}.
\]

\smallskip
\noindent\emph{Step 1: the cases $k=1$ and $k=2$.}
If $k=1$, then \eqref{eq:sqroot-reduced} is a congruence modulo~$2$.
As $c$ is odd one has $c\equiv 1\pmod 2$, hence there is exactly one solution $y\equiv 1\pmod 2$.
If $k=2$, then the squares modulo~$4$ are exactly $\{0,1\}$.
Thus solvability is equivalent to $c\equiv 1\pmod 4$, and in this case $y\equiv 1,3\pmod 4$ are the two solutions.

\smallskip
\noindent\emph{Step 2: the case $k\ge 3$ (odd squares).}
If $k\ge 3$, any odd square is congruent to $1\pmod 8$; hence solvability forces $c\equiv 1\pmod 8$.

Conversely, assume $c\equiv 1\pmod 8$.
Let $U_k=(\Z/2^k\Z)^\times$ and consider the square map $s:U_k\to U_k$, $s(y)=y^2$.
We show that $c\in s(U_k)$.
For $k=3$ this is immediate, since every element of $U_3$ squares to $1$ modulo~$8$.

Assume by induction that $y^2\equiv c\pmod{2^k}$ for some odd $y$ and write $y^2=c+2^k t$.
Choose $\varepsilon\in\{0,1\}$ with $\varepsilon\equiv t\pmod 2$ and set $y':=y+\varepsilon 2^{k-1}$.
Then
\begin{align*}
  (y')^2
  &=y^2+\varepsilon 2^k y+\varepsilon^2 2^{2k-2}\\
  &\equiv c+2^k(t+\varepsilon y)\pmod{2^{k+1}},
\end{align*}
and since $k\ge 3$ the term $2^{2k-2}$ vanishes modulo $2^{k+1}$.
As $y$ is odd, one has $t+\varepsilon y\equiv t+\varepsilon\equiv 0\pmod 2$ by construction, hence $(y')^2\equiv c\pmod{2^{k+1}}$.
This shows inductively that $c$ is a square in $U_k$ for all $k\ge 3$.

\smallskip
\noindent\emph{Step 3: the number of odd square roots.}
For $k\ge 3$ and $c\equiv 1\pmod 8$, the number of solutions to $y^2\equiv c\pmod{2^k}$ equals $\#\ker(s)$.
The kernel consists of those $y\in U_k$ with $y^2\equiv 1\pmod{2^k}$.

If $y^2\equiv 1\pmod{2^k}$ then $2^k\mid (y-1)(y+1)$.
Since $y$ is odd, the integers $y-1$ and $y+1$ are consecutive even numbers, hence $\gcd(y-1,y+1)=2$ and therefore
\[
  \min\{v_2(y-1),v_2(y+1)\}=1.
\]
Moreover,
\[
  v_2(y-1)+v_2(y+1)=v_2\bigl((y-1)(y+1)\bigr)\ge k,
\]
so one of the two factors has $2$-adic valuation at least $k-1$.
Equivalently,
\[
  y\equiv 1 \pmod{2^{k-1}}
  \qquad\text{or}\qquad
  y\equiv -1 \pmod{2^{k-1}}.
\]

Each of the two classes modulo $2^{k-1}$ has exactly two lifts modulo $2^k$, hence $\#\ker(s)=4$.
Therefore, for $k\ge 3$ and $c\equiv 1\pmod 8$, there are exactly four solutions $y\bmod 2^k$.

Combining the three cases ($k=1$, $k=2$, $k\ge 3$) and multiplying by the factor $2^j$ of lifts gives the stated counts for $x$.
\end{proof}

\begin{corollary}[Type I block representation counts]\label{cor:typeI-counts}
Let $L=\langle 2^a u\rangle$ with $a\ge 0$ and $u\in\Ztwo^\times$.
Let $t\in\Ztwo$.
If $v_2(t)<a$ then $r_m(t;L)=0$ for all $m>a$.
If $v_2(t)\ge a$, write $t=2^a t_a$.
For any $m>a$ one has
\[
  r_m(t;L)=2^a\cdot\#\{x\bmod 2^{m-a}: x^2\equiv u^{-1}t_a\!\!\pmod{2^{m-a}}\},
\]
and the right-hand side is given explicitly by Lemma~\ref{lem:square-roots}.
\end{corollary}

\begin{proof}
By definition, $r_m(t;L)$ counts solutions $x\bmod 2^m$ to the congruence
\[
  2^a u\,x^2\equiv t\pmod{2^m}.
\]
If $m>a$ and $v_2(t)<a$, then the left-hand side is divisible by $2^a$ while the right-hand side is not, hence there are no solutions.

Assume $v_2(t)\ge a$ and write $t=2^a t_a$.
Dividing the congruence by $2^a$ yields
\[
  u\,x^2\equiv t_a\pmod{2^{m-a}}.
\]
This depends only on the class of $x$ modulo $2^{m-a}$.
Each residue class modulo $2^{m-a}$ has exactly $2^a$ lifts to a residue class modulo $2^m$.
Thus
\[
  r_m(t;L)=2^a\cdot\#\{x\bmod 2^{m-a}: x^2\equiv u^{-1}t_a\pmod{2^{m-a}}\},
\]
as claimed.
The remaining count is given by Lemma~\ref{lem:square-roots}.
\end{proof}


%% file: main.bbl
\begin{thebibliography}{1}

\bibitem{choYamauchi2020}
S.~Cho and T.~Yamauchi.
\newblock A reformulation of the {S}iegel series and intersection numbers.
\newblock {\em Mathematische Annalen}, 377(3--4):1757--1826, 2020.

\bibitem{ganYu2000}
W.~T. Gan and J.-K. Yu.
\newblock Group schemes and local densities.
\newblock {\em Duke Mathematical Journal}, 105(3):497--524, 2000.

\bibitem{grossKeating1993}
B.~H. Gross and K.~Keating.
\newblock On the intersection of modular correspondences.
\newblock {\em Inventiones Mathematicae}, 112(2):225--245, 1993.

\bibitem{ikedaKatsurada2018}
T.~Ikeda and H.~Katsurada.
\newblock On the {G}ross--{K}eating invariant of a quadratic form over a
  non-{A}rchimedean local field.
\newblock {\em American Journal of Mathematics}, 140(6):1521--1565, 2018.

\bibitem{kitaokaBook}
Y.~Kitaoka.
\newblock {\em Arithmetic of Quadratic Forms}, volume 106 of {\em Cambridge
  Tracts in Mathematics}.
\newblock Cambridge University Press, Cambridge, 1993.

\bibitem{omearaQuad}
O.~T. O'Meara.
\newblock {\em Introduction to Quadratic Forms}, volume 117 of {\em Die
  Grundlehren der mathematischen Wissenschaften}.
\newblock Springer-Verlag, Berlin--G{"o}ttingen--Heidelberg, 1963.

\bibitem{yang1998}
T.~Yang.
\newblock An explicit formula for local densities of quadratic forms.
\newblock {\em Journal of Number Theory}, 72(2):309--356, 1998.

\end{thebibliography}
